\documentclass[reqno]{amsart}
\usepackage{amsfonts, amsmath, amsthm, amssymb, latexsym, xfrac, mathrsfs, xcolor}

\makeatletter
\@namedef{subjclassname@2020}{\textup{2020} Mathematics Subject Classification}
\makeatother

\theoremstyle{plain}
\newtheorem{theorem}{Theorem}[section]
\newtheorem{corollary}[theorem]{Corollary}
\newtheorem{lemma}[theorem]{Lemma}
\newtheorem{proposition}[theorem]{Proposition}
\theoremstyle{definition}
\newtheorem{definition}[theorem]{Definition}
\theoremstyle{remark}
\newtheorem{remark}[theorem]{Remark}
\newtheorem{remarks}[theorem]{Remarks}

\numberwithin{equation}{section}

\dedicatory{In memory of Gerrit van Dijk}

\title[Limits of Bessel functions as the rank tends to infinity]
{Limits of Bessel functions for root systems as the rank tends to infinity }
\author{Dominik Brennecken and Margit R\"osler} 
\address{Institut f\"ur Mathematik, Universit\"at Paderborn, Warburger Str. 100, D-33098 Paderborn, Germany}
\email{bdominik@math.upb.de,  roesler@math.upb.de}


\begin{document}

\begin{abstract} We study the asymptotic behaviour of Bessel functions associated to root systems of type $A_{n-1}$ and type $B_n$ with positive multiplicities as the rank $n$ tends to infinity. In both cases, we characterize the possible limit functions and the Vershik-Kerov type sequences of spectral parameters for which such limits exist. 
In the type $A$ case, this gives a new and very natural approach to recent results by Assiotis and Najnudel in the context of $\beta$-ensembles in random matrix theory. These results
generalize known facts about the approximation of the positive-definite Olshanski spherical functions of the space of infinite-dimensional Hermitian matrices over $\mathbb F = \mathbb R, \mathbb C, \mathbb H $ (with the action of the associated infinite unitary group) by spherical functions of finite-dimensional spaces of Hermitian matrices.  In the type B case, our results include asymptotic results for the spherical functions associated with the Cartan motion groups of non-compact Grassmannians as the rank goes to infinity, and  a classification of the Olshanski spherical functions of the associated inductive limits. 
\end{abstract}

\maketitle

\section{Introduction}

The asymptotic analysis of multivariate special functions  has a long tradition in infinite dimensional harmonic analysis, tracing back to the work of Olshanski, Vershik, and Kerov, see \cite{Ol90, OV96, VK82}. 
Of particular interest in this context are the behaviour of spherical representations and the limits of spherical functions of increasing families of Gelfand pairs as specific dimensions tend to infinity.  

Bessel functions associated with root systems generalize the spherical functions of Riemannian symmetric spaces of Euclidean type, which occur for special values of the multiplicity parameters. 
They appear naturally in rational Dunkl theory, with an intimate connection to  the Dunkl kernel and the associated harmonic analysis. We refer to \cite{Op93} for a general treatment of such Bessel functions and to \cite{Ro03, dJ06, RV08, DX14} for an overview of rational Dunkl theory including the connection with symmetric spaces.
There are two classes of particular interest, including applications to  $\beta$-ensembles in random matrix theory, namely those of type $A_{n-1}$ and type $B_n$.
We refer to \cite{Fo10} for a general background and to \cite{BCG22} for some recent developments.   In the  cases of type $A$ and $B,$ the Bessel functions  can be expressed as hypergeometric series involving Jack polynomials, c.f. Section \ref{Bessel}.   Bessel functions of type $A_{n-1}$ have a continuous multiplicity parameter $k \geq 0$ and include as special cases the spherical functions of  
the motion groups  $U_n(\mathbb F)\ltimes \text{Herm}_n(\mathbb F)$ over $\mathbb F = \mathbb R, \mathbb C$ or $\mathbb H$, where the unitary group $U_n(\mathbb F)$ acts by conjugation on the space $\text{Herm}_n(\mathbb F)$ of Hermitian matrices over $\mathbb F$. These cases correspond to $k = \frac{d}{2}$ with $d = \text{dim}_\mathbb R \mathbb F \in \{1, 2 ,4\}.$ Bessel functions of type $B_q$ have non-negative multiplicity parameters of the form $\kappa = (k^\prime, k)$, with $k$ the multiplicity  on the roots $\pm(e_i\pm e_j)$ and $k^\prime$ that on the roots $\pm e_i$. They generalize the spherical  functions of the motion groups  $(U_p(\mathbb F) \times U_q(\mathbb F))\ltimes M_{p,q}(\mathbb F),$ with  $p\geq q.$  Here the multiplicities are $k = \frac{d}{2}, \, k^\prime = \frac{d}{2}(p-q+1) - \frac{1}{2}.$
In \cite{RV13}, limits of the spherical functions 
of these motion groups as $p\to \infty$ and the associated Olshanski spherical pairs were studied, where the rank $q$ remained fixed.


In the present paper, we study Bessel functions of type $A_{n-1}$ and type $B_n$ with arbitrary positive multiplicities and imaginary spectral parameters as the rank tends to infinity. We characterize the sequences of spectral parameters for which a limit exists, and we uniquely parametrize the sets of possible limit functions.

In the group cases, the obtained limits are 
exactly the (positive definite) Olshanski spherical functions of the associated infinite dimensional motion groups $U_\infty(\mathbb F) \ltimes \text{Herm}_\infty(\mathbb F)$ and $(U_\infty(\mathbb F)\times U_\infty(\mathbb F)) \ltimes M_{\infty}(\mathbb F),$ which are obtained as inductive limits of the above finite dimensional motion groups. These Olshanski spherical functions were already determined by Pickrell \cite{Pi91} (Section 5). In the case of $U_\infty(\mathbb C) \ltimes \text{Herm}_\infty(\mathbb C),$ their approximation by positive definite spherical functions of the corresponding finite-dimensional Cartan motion groups is a classical result of Olshanski and Vershik  \cite{OV96}, 
where the limits showed up as the characteristic functions of the ergodic measures on the space of  infinite complex Hermitian matrices with respect to the action of $U_\infty(\mathbb C).$ The same method was used in \cite{Ra08} to obtain  approximations of the Olshanski spherical functions of $(U_\infty(\mathbb C)\times U_\infty(\mathbb C)) \ltimes M_{\infty}(\mathbb C)$ by spherical functions of $(U_n(\mathbb C) \times U_n(\mathbb C))\ltimes M_n(\mathbb C).$ 
 In \cite{Bo07}, the limit results of \cite{OV96} were extended,   by somewhat different methods, to $\mathbb F= \mathbb R$ and $\mathbb H.$ In the context of $\beta$-ensembles from random matrix theory, Assiotis and Najnudel \cite{AN21} extended the approximation results of \cite{OV96} for the group cases of type $A$ to 
 Bessel functions of type $A$ with arbitrary positive multiplicity parameter $k = \beta/2.$  Their approach is probabilistic and again different from that of \cite{OV96}, and it provides a larger, complex domain of convergence in the approximation.  

For Bessel functions of type $A$,  we mainly recover the results of Assiotis and Najnudel \cite{AN21}. However, we follow the method 
of \cite{OV96} and obtain very natural direct proofs based on the expansion of the Dunkl-type Bessel functions in terms of Jack polynomials. These results,  in the spirit of the work of Okounkov and Olshanski \cite{OO98}, are contained 
in Theorem \ref{main}, the main result of Section \ref{type A}.  

To become  precise, 
we consider the Bessel functions $J_{A_{n-1}}(i\lambda(n),(z, 0, \ldots, 0))$  with fixed multiplicity $k>0$ and $z \in \mathbb C^r$ for sequences of spectral parameters $\lambda(n) \in \mathbb R^n$ as $n\to \infty. $
Following \cite{OO98, Fa08, AN21}, we characterize those sequences $(\lambda(n))_{n\in \mathbb N}$ for which the associated sequence of Bessel functions converges, in terms of specific  real parameters $\alpha = (\alpha_i)_{i\in \mathbb N},\, \beta, \gamma$ with $\gamma\geq 0.$ These parameters describe the growth of $(\lambda(n))$ (a so-called Vershik-Kerov-sequence) as $n \to \infty.$ 
We obtain that 
\begin{equation}\label{limitres_A} \lim_{n\to \infty} J_{A_{n-1}}(i\lambda(n),z) \,=\, \prod_{j=1}^\infty e^{i\beta z_j - \frac{\gamma}{2k}z_j^2} \prod_{l=1}^\infty \frac{e^{-i\alpha_l z_j}}{\bigl(1- \frac{i\alpha_l z_j}{k}\bigr)^k}\,,
\end{equation}
where for each $r\in \mathbb N,$ the convergence is locally uniform
 in the complex domain 
$$ \Big\{ z\in \mathbb{C}^{r} : \|\text{Im}\,z\|_\infty <\frac{k}{r|\alpha_1|} \Big\}.$$
This coincides with the result of \cite{AN21}. Our proof becomes even simpler if all entries of $\lambda(n)$ are non-negative, which is described in Remark \ref{nonneg}. Moreover, in Proposition \ref{9.12} we uniquely characterize the set of possible limit functions in \eqref{limitres_A}. 
 In the group cases $k = \frac{d}{2}$,  the limit functions  are products of Polya functions (in the sense of \cite{Fa08}) and determine the positive-definite Olshanski spherical functions of the
   pairs $(U_\infty(\mathbb F)\ltimes \text{Herm}_\infty(\mathbb F), U_\infty(\mathbb F)),$ c.f. \cite{Pi91}. They are approximated, uniformly on compact sets,  by sphercial functions of the associated finite-dimensional Gelfand pairs.

\medskip
Our approach for Bessel functions of type $A$ extends to the type $B$ case in a natural way, where the Bessel functions also have an explicit Jack polynomial expansion. In this case, matters are even a bit easier, because one may restrict to non-negative Vershik-Kerov sequences, contained in Weyl chambers of type $B$. This is done in Section \ref{type B}. 
To become precise again, we consider Bessel functions $J_{B_n}(\kappa_n, i\lambda(n),(z,0, \ldots, 0))\,$ for $n \to \infty,$ where $\lambda(n) \in \mathbb R^n$ and the 
multiplicity is of the form $\kappa_n = (k_n^\prime, k)$ with $k_n \geq 0, k>0.$ Here the first multiplicity parameter may also vary with $n$, which is motivated by the geometric cases. 
Again, we characterize the sequences of spectral parameters for which a limit exists; the condition now being that the sequence $\bigl(\frac{\lambda(n)^2}{\nu_n}\bigr)_{n\in \mathbb N}$  has to be Vershik-Kerov, where $\nu_n = k_n^\prime +k(n-1) + \tfrac{1}{2}.$
We obtain limits of the form
$$
\lim_{n\to \infty} J_{B_{n}}(\kappa_n; i\lambda(n),z) \, =\,  \prod_{j=1}^\infty e^{-\frac{\beta z_j^2}{4}} \prod_{l=1}^\infty \frac{e^{\frac{\alpha_l z_j^2}{4}}}{\bigl(1+ \frac{\alpha_l z_j^2}{4k}\bigr)^k}\,,\qquad 
$$
with specific  non-negative parameters $\alpha = (\alpha_l)_{l\in \mathbb N} $ and $\beta$ satisfying $\sum_{l=1}^\infty \alpha_l \leq \beta. $
The convergence is uniform on compact subsets of the complex domain
$$  \Big\{ z\in \mathbb{C}^{(\infty)} : \|\text{Im}\,z\|_\infty < \, 2\,\sqrt{ \frac{k}{\alpha_1}} \,\Big\}.$$
This is contained in Theorem \ref{main_B}, the main result of Section \ref{type B}.  The possible range of parameters $\alpha, \beta$ is determined in Proposition \ref{Omega_plus}. 
It finally turns out that for $k = \frac{d}{2}$ with $d = 1,2,4$,  the above limits can be precisely identified with  the positive definite Olshanski spherical functions of spherical pairs $(G_\infty, K_\infty),$ which are obtained as inductive limits of the motion groups 
$(U_p(\mathbb F) \times U_q(\mathbb F))\ltimes M_{p,q}(\mathbb F)$
as both dimension parameters $p, q$ (with $p\geq q$)  tend to infinity. Moreover, these Olshanski spherical functions are explicitly approximated by spherical functions of the corresponding finite dimensional spaces. 
Our parametrization of the Olshanski sphercial functions 
slightly differs from that of \cite{Pi91, Ra08} (but is equivalent). 

We finally mention that certain limits of type $B$ Bessel  functions and methods from Dunkl theory also play a role 
in \cite{Xu23}, where the distribution of the singular values of sums of rectangular matrices  is studied for low and high temperatures, i.e. in the limits $k \to 0$ and $k \to \infty.$

\section*{Acknowledgements}
The authors received funding from the  Deutsche Forschungsgemeinschaft German Research Foundation (DFG), via RO 1264/4-1 and SFB-TRR  358/1 2023 - 491392403 (CRC Integral Structures in Geometry and Representation Theory). We are indebted to one of the referees for bringing the paper \cite{AN21} to our attention, and for some further helpful comments. 

\section{Bessel functions of type $A$ and type $B$ }\label{Bessel}

In this section, we recall some basic facts about Bessel 
functions in Dunkl theory. We shall not go into details, but refer the reader to  \cite{DX14,  Ro03}  for some general background on Dunkl theory, and to \cite{Op93, dJ06, RV08} for Dunkl-type Bessel functions and their relevance in the symmetric space context.
For a reduced (but not necessarily essential) root system $R \subset \mathbb R^n$ we fix a non-negative multiplicity function $k: R \mapsto [0,\infty),$ i.e. $k$ is invariant under the associated Weyl group $W.$ Let $E = E_{R}(k;\,.\,,\,.\,) : \mathbb C^n \times\mathbb C^n \to \mathbb C$ denote the associated Dunkl kernel, where $E(\lambda, \,.\,)$ is characterized as the unique analytic solution of the joint eigenvalue problem for the associated rational Dunkl operators with spectral variable $\lambda \in \mathbb C^n$, normalized according to $E(\lambda,0)=1.$  The kernel $E$ is holomorphic on $\mathbb C^n \times \mathbb C^n$ and symmetric in its arguments. 
For each $\lambda \in \mathbb R^n$, there exists a compactly supported probability measure 
$\mu_\lambda $ on $\mathbb R^n$ such that 
\begin{equation}\label{int_rep} E(i\lambda,z) = \int_{\mathbb R^n} e^{i\langle \xi, z\rangle} d\mu_{\lambda}(\xi) \quad \text{ for all } z \in \mathbb C^n,\end{equation}
where $\langle w,z \rangle = \sum_{j=1}^n w_jz_j\,.$ Let us emphasize that the existence of such an integral representation hinges on the non-negativity of  $k.$ 
 If $k=0,$ then $E(i\lambda, z) = e^{i\langle \lambda, z\rangle}.$
The Bessel function associated with $R$ and $k$ is defined by
$$ J(\lambda,z)  = \frac{1}{|W|} \sum_{w\in W} E(\lambda, wz).$$
It is $W$-invariant in both arguments. Note that in view of \eqref{int_rep}, \begin{equation}\label{J_estimate} \vert J(i\lambda,z)\vert \leq \, J(\lambda, - {\rm Im} \,z) \quad \text{ for } \lambda \in \mathbb R^n.\end{equation}

 We shall be concerned with the root systems $ A_{n-1}= \{\pm(e_i-e_j): 1 \leq i < j \leq n\}\subset \mathbb R^n$ and $B_n= \{\pm e_i, \pm(e_i\pm e_j): 1\leq i < j \leq n\}\subset \mathbb R^n, $ where $(e_i)$ denotes the standard basis of $\mathbb R^n.$ In both
cases, the Bessel functions can be written as hypergeometric series in terms of Jack polynomials. 
For $A_{n-1},$ the multiplicity function is given by a single parameter $ k \in [0, \infty).$ 
We write $\Lambda_n^+$ for the set of partitions $\kappa = (\kappa_1, \kappa_2, \ldots \,)$ of length $l(\kappa)\leq n$ and denote by  $C_\kappa^{(n)}, \, \kappa \in \Lambda_n^+$  the (symmetric) Jack polynomials in $n$ variables of index $\alpha = \frac{1}{k} \in [0, \infty],$ normalized such that
$$ \sum_{|\kappa|= m} C_\kappa^{(n)}(z) = (z_1 + \ldots + z_n)^m, \quad m \in \mathbb N_0.$$
The Jack polynomials are stable with respect to the number of variables, i.e. for $\kappa \in \Lambda_{r}^+$ with $r<n$ we have
\begin{equation}\label{stability} C_\kappa^{(r)}(z_1, \ldots, z_{r}) =\begin{cases} 
	C_{\kappa^\prime}^{(n)}(z_1, \ldots, z_{r}, \underline 0_{n-r}) & \text{ if } \kappa^\prime = (\kappa, 0,...)\\
	\,0 & \text{ otherwise};
	\end{cases}\end{equation}
with the notation $\underline a_{j}:= (a, \ldots, a)\in \mathbb C^j$ for $a\in \mathbb C.$   See \cite[Prop. 2.5]{St89}  together with \cite[formula(16)]{Ka93}.	
Therefore the Jack polynomials $C_\kappa^{(n)}$ uniquely extend to continuous functions
$C_\kappa$  on $\,\mathbb C^{(\infty)} = \bigcup_{n=1}^\infty \mathbb C^n, $ equipped with the inductive limit topology. We shall often consider elements from $\mathbb C^{(\infty)}$ as sequences $x=(x_n)_{n\in \mathbb N}$ in $\mathbb C$ with $x_n \not=0$ for at most finitely many $n;$ for $\mathbb R^{(\infty)}$ accordingly. 

By \cite{BF98}, the Bessel function of type $A_{n-1}$ with multiplicity $k$ is given by the Jack hypergeometric series
\begin{equation} \label{Bessel_A} J_{A_{n-1}}(\lambda,z) = \, _0F_0(\lambda,z)\,=\, \sum_{\kappa\in \Lambda_n^+} \frac{C_\kappa(\lambda) C_\kappa(z)}{|\kappa|!\, C_\kappa(\underline 1_n)}\end{equation}
with the Jack polynomials of index $\alpha = 1/k$ as above.

Bessel functions of type $A_{n-1}$ occur as the (zonal) spherical functions of 
the Gelfand pair $(G_n, K_n):= (U_n(\mathbb F) \ltimes \text{Herm}_n(\mathbb F), U_n(\mathbb F)),$  where the unitary group $U_n(\mathbb F)$ over $\mathbb F = \mathbb R, \mathbb C, \mathbb H$ acts by conjugation on the space  $\text{Herm}_n(\mathbb F)$ of Hermitian matrices over $\mathbb F.$ 
Recall that the spherical functions of a Gelfand pair $(G,K)$ 
can be characterized as the continuous, $K$-biinvariant, non-zero  functions $ \varphi$ on  $G$ satisfying the product formula
$$ \varphi(g) \varphi(h) = \int_K \varphi(gkh)dk \quad (g,h \in G).$$
The following characterization is possibly folklore, but not well documented in the literature. For the reader's convenience, we therefore provide a proof. 

\begin{lemma} \label{A_spherical_char} The spherical functions of $(G_n, K_n),$ considered as  $U_n(\mathbb F)$-invariant functions on $\text{Herm}_n(\mathbb F),$ are given by the Bessel functions 
$$ \varphi_\lambda(X) = J_{A_{n-1}}\bigl(\tfrac{d}{2}; i\lambda, \sigma(X)\bigr), \quad \lambda \in \mathbb C^n $$
where $d = \text{dim}_{\mathbb R}\mathbb F$ and 
$\sigma(X) \in \mathbb R^n$ denotes the eigenvalues of $X\in \text{Herm}_n(\mathbb F)$, decreasingly ordered by size and counted according to their multiplicity. Moreover, $\varphi_\lambda = \varphi_\mu$ iff there exists some $w \in S_n$ with $\mu = w.\lambda$,  and $\varphi_\lambda$ is positive definite iff $\lambda \in \mathbb R^n.$
\end{lemma}

\begin{proof}
Consider the Gelfand pair $(\widetilde G_n, \widetilde K_n) = (SU_n(\mathbb F) \ltimes \text{SHerm}_n(\mathbb F), SU_n(\mathbb F)).$ Note that $\widetilde G_n$ is the Cartan motion group of $SL_n(\mathbb F),$ which is connected and semisimple. Thus by \cite[Sect.6]{dJ06} (c.f. also \cite[Sect.3]{RV08}), the spherical functions of $(\widetilde G_n, \widetilde K_n)$ are given, as functions on $\text{SHerm}_n(\mathbb F),$ by 
$$ \widetilde\varphi_\lambda(X) = J_{A_{n-1}}\bigl(\tfrac{d}{2}; i\lambda, \sigma(X)\bigr), \quad \lambda \in \mathbb C_0^n:=\{z\in \mathbb C^n: z_1 + \ldots + z_n = 0\}. $$
Consider the mapping $\,\pi:\mathbb C^n \to \mathbb C_0^n,\, z \mapsto z - \frac{1}{n}\langle z, \underline 1_n\rangle \underline 1_n$, where again the standard inner product $\langle \, .\,, \, .\, \rangle$ on $\mathbb R^n$ is extended to $\mathbb C^n\times \mathbb C^n$ in a bilinear way. The restriction of $\pi$ to $\mathbb R^n$ is the orthogonal projection onto $\mathbb R_0^n = \mathbb R^n \cap \mathbb C_0^n.$ Then for $z,w \in \mathbb C^n$ and arbitrary multiplicity $k \geq 0,$ 
\begin{equation} \label{reduktiv} J_{A_{n-1}}(k;z,w)= e^{\langle z, \underline 1_n\rangle \langle w, \underline 1_n\rangle/n} \cdot J_{A_{n-1}}(k;\pi(z), \pi(w)).\end{equation}
This follows e.g. from \cite[Propos. 3.19]{BF98}.
Now suppose $\psi $ is a spherical function of $(G_n, K_n),$ considered as an $S_n$-invariant function on $\mathbb R^n.$ Then
$$ \psi (x) \psi (y) = \int_{K_n} \psi(x + kyk^{-1})d_nk,$$
where $x, y\in \mathbb R^n$ are identified with the corresponding $n\times n$-diagonal matrices. It follows that
$$ \psi(x) = \psi(x-\pi(x) + \pi(x)) = \psi(x-\pi(x)) \cdot \psi(\pi(x)) = e^{\alpha \langle x, \underline 1_n\rangle}\cdot \psi(\pi(x)),$$ 
where $\alpha \in \mathbb C$ is a constant and the restriction of $\psi$ to $\mathbb R_0^n$ is spherical for $(\widetilde G_n, \widetilde K_n).$ Conversely, it is easily checked that each spherical function of $(\widetilde G_n, \widetilde K_n)$ extends to a spherical function of $(G_n, K_n)$ in this way. 
The assertion now follows from \eqref{reduktiv}. The assertion concerning the positive-definite spherical functions follows from \cite[Theorem 5.4]{Wo06}. 
\end{proof}

For the root system $R=B_n$, we denote the multiplicity (by  slight abuse of notation) as $(k^\prime, k),$  where $k$ is the value on the roots $\pm (e_i \pm e_j)$ and $k^\prime $ is the value on $\pm e_i.$ Put $\nu_n := k^\prime + k(n-1)+ \frac{1}{2}\,.$ Then the Bessel function of type $B_n$ with multiplicity $(k^\prime, k)$ can be written as 
\begin{equation} \label{Bessel_B} 
J_{B_n}(\lambda, z) 	= \, _0F_1\bigl(\nu_n; \tfrac{\lambda^2}{2}, \tfrac{z^2}{2}\bigr) \, = \, \sum_{\kappa \in \Lambda_n^+} \frac{1}{4^{|\kappa|}[\nu_n]_\kappa}   \frac{C_\kappa(\lambda^2) C_\kappa(z^2)}{|\kappa|!\, C_\kappa(\underline 1_n)}
\end{equation}
with the hypergeometric series 
$$ _0F_1\bigl(\nu; \lambda, z\bigr)\, = \,\sum_{\kappa \in \Lambda_n^+} \frac{1}{[\nu]_\kappa} 
 \frac{C_\kappa(\lambda) C_\kappa(z)}{|\kappa|!\, C_\kappa(\underline 1_n)}\,.$$
Here the squares in the arguments are understood componentwise and again, the Jack polynomials are those of index $\alpha= 1/k.$ 
It is easily seen that both $\, _0F_0$ and $_0F_1$  converge locally uniformly on $\mathbb C^n\times \mathbb C^n$; c.f. \cite{BR23} for precise convergence properties of Jack hypergeometric series.

Bessel functions of type $B$ occur as the spherical functions of 
the Gelfand pairs $(G, K)$ with 
$$ G=(U_p(\mathbb F)\times  U_q(\mathbb F))\ltimes M_{p,q}(\mathbb F), \> K = U_p(\mathbb F)\times  U_q(\mathbb F), \quad p \geq q,$$
 where $M_{p,q}(\mathbb F)$ is the space of $p\times q$ matrices over $\mathbb F = \mathbb R, \mathbb C, \mathbb H$ and $K$ acts on $M_{p,q}(\mathbb F)$ via 
 $(U,V).X = UXV^{-1}.$ The group $G$ is the Cartan motion group of the non-compact Grassmann manifold $U(p,q;\mathbb F)/U_p(\mathbb F) \times U_q(\mathbb F)$ which is of rank $q$. 
  The spherical functions of $(G,K)$ may be considered as $K$-invariant functions on $M_{p,q}(\mathbb F)$ and thus depend only on the singular values of their argument. Again as a consequence of \cite{dJ06},  they are given by the Bessel functions 
 $$ \varphi_\lambda(X) = J_{B_{q}}(\kappa, i\lambda, \sigma_{sing}(X)), \> \lambda \in \mathbb C^q,$$
 where $ \kappa = (k^\prime, k)= \bigl(\frac{d}{2}(p-q+1) - \frac{1}{2}, \frac{d}{2}\bigr)$ and $\sigma_{sing}(X) = \sigma(\sqrt{X^*X}) \in \mathbb R^q $ denotes the set of singular values of $X\in M_{p,q}(\mathbb F),$ ordered by size. 
 We may therefore also consider the $\varphi_\lambda$ as functions on the closed Weyl chamber 
 \begin{equation}\label{chamber}\overline C_q = \{x = (x_1,\ldots, x_q) \in \mathbb R^q: x_1 \geq \ldots \geq x_q \geq 0\},\end{equation}
i.e.  $ \varphi_\lambda(x) = J_{B_q}(\kappa;i\lambda, x), \, x \in \overline C_q$. Moreover, $\varphi_\lambda = \varphi_\mu$ iff there exists some $w \in W=S_q \ltimes \mathbb Z_2^q$  with $\mu = w.\lambda.$ 

 The positive-definite spherical functions  are 
 the $\varphi_{\lambda}$ with $\mathbb R^q$, which again follows from \cite[Theorem 5.4]{Wo06}.

\section{The type $A$ case}\label{type A}

We start with some motivation from asymptotic spherical harmonic analysis, see \cite{Ol90, Fa08} for a general background.  Suppose that $(G_n, K_n), \,n\in \mathbb N$ is an increasing sequence of Gelfand pairs, where $G_n \subseteq G_{n+1}, \,K_n \subseteq K_{n+1}$ are closed subgroups satisfying $K_n = G_n \cap K_{n+1}.$ 
Then the pair $(G_\infty, K_\infty)$ with the  inductive limit groups $G_\infty := \lim_{n\rightarrow} G_n, \, K_\infty:= \lim_{n\rightarrow} K_n$ is called an Olshanski spherical pair. 
The spherical functions of  $(G_\infty, K_\infty)$ are defined as the  continuous, non-zero  and $K_\infty$-biinvariant functions $\varphi:  G_\infty \to \mathbb C$ satisfying 
	$$ \varphi (g) \varphi(h) = \lim_{n\to \infty} \int_{K_n} \varphi (gkh)d_nk \quad (g,h \in G_\infty),$$
 where $d_nk$ is the normalized Haar measure on $K_n$. We remark that this definition is according to \cite{Fa08}, whereas in \cite{Ol90}  spherical functions are in addition required to be positive definite. 
 Consider now the sequence of Gelfand pairs $(G_n, K_n)= (U_n(\mathbb F) \ltimes \text{Herm}_n(\mathbb F), U_n(\mathbb F))$ as above.  We regard $G_n$ and $K_n$ as closed subgroups of $G_{n+1}$ and $ K_{n+1}$ in the usual way. Then 
$\, (G_\infty, K_\infty)$ with the inductive limits 
\begin{equation}\label{Olshanski_A} K_\infty :=  \lim_{\rightarrow} K_n \,= 
U_\infty(\mathbb F), \quad G_\infty := \lim_{\rightarrow} G_n\, = U_\infty(\mathbb F)\ltimes \text{Herm}_\infty(\mathbb F)\end{equation}
 is an Olshanski spherical pair. 
The positive definite spherical functions of $(G_\infty, K_\infty)$ were  completely determined by Pickrell \cite[Sect.5]{Pi91}, see also \cite{OV96} for $\mathbb F = \mathbb C$,  and \cite[Section 3]{Fa08}. As functions on $\text{Herm}_\infty(\mathbb F),$ they are given by
$$ \varphi(X) = \prod_{j=1}^\infty e^{i\beta x_j - \frac{\gamma}{d} x_j^2} \prod_{l=1}^\infty \frac{e^{-i\alpha_lx_j}}{(1-i\frac{2}{d}\alpha_l x_j)^{d/2}}\,,$$
where $\beta, \gamma \in \mathbb R, \gamma \geq 0,  \, \alpha_l \in \mathbb R$ with $\sum_{l=1}^\infty \alpha_l^2 < \infty, $ and $(x_1, x_2, \ldots\,)\in \mathbb R^{(\infty)}$ are the eigenvalues of $X,$ ordered by size and counted according to their multiplicity.  The product is invariant under rearrangements of  the $\alpha_l$. For $\mathbb F=\mathbb C$ it is also noted in \cite{OV96} that 
the set of positive definite spherical functions is bijectively parametrized by the set 
$$\{(\alpha, \beta,\gamma): \beta \in \mathbb R, \gamma \geq 0, \alpha = \{\alpha_1, \alpha_2, \ldots \} \text{ a multiset with $\alpha_l \in \mathbb R$ and } \sum_l \alpha_l^2 < \infty\}.$$ 
In \cite{OV96}, explicit approximations of the positive definite spherical functions by positive definite spherical functions of the pairs $(G_n, K_n)$ with $n \to \infty$ by use  of spherical expansions were obtained in the case $\mathbb F = \mathbb C.$ In \cite{Bo07}, this was generalized  by  different methods to $\mathbb F = \mathbb R, \mathbb H.$  
 
 In the present paper, we shall extend the approach of \cite{OV96} to obtain limits for Bessel functions of type $A_{n-1}$ with an arbitrary multiplicity parameter $k>0.$

Let us first turn to the spectral parameters to be considered for $n \to \infty.$ Instead of working  with multisets, it will be convenient for us to work with sequences (or finite tuples) with a prescribed order of their components.  We introduce the following order  on $\mathbb R$:
$$ x \ll y \, \text{ iff either } |x|<|y| \text{ or } |x|=|y| \text{ and } x\leq y.$$
For instance, the sequence $(3,-3,2,1,-1,-1,0,0, \ldots)$ is decreasing w.r.t. $\ll.$

\begin{definition}\label{VK}
Consider  $\lambda(n)\in \mathbb R^n$ such that its entries are decreasing with respect to $ \ll$. 
	We regard  $(\lambda(n))_{n\in \mathbb N}$ as a sequence in $\mathbb R^{(\infty)} = \bigcup_{n=1}^\infty \mathbb R^n$ and call it a Vershik-Kerov sequence (VK sequence for short), if the following limits exist:
\begin{align*} \alpha_i &:= \lim_{n\to \infty} \frac{\lambda(n)_i}{n} \quad (i \in \mathbb N),\\
\beta &:= \lim_{n\to \infty} \frac{p_1(\lambda(n))}{n}\,,\\
\delta &:= \lim_{n\to \infty} \frac{p_2(\lambda(n))}{n^2},	
\end{align*}
where 
$$ p_m(x) = \sum_{i=1}^\infty x_i^m \quad \text{for } m \in \mathbb N, \,\,  p_0  \equiv 1 $$
are the power sum symmetric functions on $\mathbb R^{(\infty)}.$ They generate the algebra of
symmetric functions on $\mathbb R^{(\infty)}$, i.e. the symmetric polynomial functions in arbitrary many variables. 
	\end{definition}

	\begin{remark} 
	Our definition of a VK sequence is equivalent to the Olshanski-Vershik conditions of \cite[Def 2.2]{AN21}, which are slightly weaker than the conditions of \cite{OO98}. 
		\end{remark}

	\begin{lemma}\label{VK_Lemma} \begin{enumerate}\itemsep=1pt
	\item[{(i)}] If $(\lambda(n))_{n\in \mathbb N} $ is a Vershik-Kerov sequence with associated parameters $(\alpha_i), \beta, \delta$ as above, then $$\, \, \gamma:= \delta - \sum_{i=1}^\infty \alpha_i^2 \,\geq 0.$$  
	In particular, the sequence $\alpha = (\alpha_i)_{i\in \mathbb N}$ is square-summable. 
		\item[{(ii)}] If in addition $\lambda(n)_i \geq 0$ for all $i, n \in \mathbb N,$ then $\,\gamma=0.$
	\end{enumerate}
	\end{lemma}

	\begin{definition} Suppose that $(\lambda(n))_{n\in \mathbb N} $ is a VK sequence. Then the triple $\omega= (\alpha, \beta, \gamma) $ with $ \alpha = (\alpha_i)_{i\in \mathbb N}$ are called the VK parameters of the sequence $(\lambda(n))_{n\in \mathbb N}.$ 
	Note that the entries of $\alpha$ are also ordered w.r.t. $\ll.$
\end{definition}

\begin{proof}[Proof of Lemma \ref{VK_Lemma}] For fixed $N \in \mathbb N$ and all $n \geq N$ one has
$$ \sum_{i=1}^N \alpha_i^2 \, \leq \, 
\sum_{i=1}^N \left(\alpha_i^2 - \frac{\lambda(n)_i^2}{n^2}\right) \, + \,  \sum_{i=1}^n \frac{\lambda(n)_i^2}{n^2}\,.$$
By definition of  $(\alpha_i)$ and $\delta,$ the right-hand side tends to $\delta$ as $n \to \infty.$ This proves part (i). 	

(ii) By the ordering of the entries of $\lambda(n),$ we obtain for $N \in \mathbb N$ and $n \geq N $ that
 \begin{align*}
	\frac{p_2(\lambda(n))}{n^2} \, = &
\, \sum_{i=1}^{N-1} \left(\frac{\lambda(n)_i}{n}\right)^2 \, +\, \sum_{i=N}^{n} \left(\frac{\lambda(n)_i}{n}\right)^2 \\
\leq  & \,\sum_{i=1}^{N-1}\left(\frac{\lambda(n)_i}{n}\right)^2 \,
+\,\frac{\lambda(n)_N}{n} \sum_{i=1}^{n} \frac{\lambda(n)_i}{n}\,.
\end{align*}   
Taking the limit $n\to\infty$ on both sides, we obtain that
$$ \delta\, \leq  \, \sum_{i=1}^{N-1} \alpha_i^2 \,+\,
\alpha_N\beta.$$
As $\lim_{N\to\infty}\alpha_N =0$, this implies that
$\, \delta\leq \sum_{i=1}^{\infty} \alpha_i^2\,$ and therefore $\gamma =0.$  
\end{proof}

We shall throughout fix a strictly positive multiplicity $k>0$ on $A_{n-1}$ and suppress it in our notation. 

For sequences $(\lambda(n))_{n\in \mathbb N}$ of spectral parameters $\lambda(n)\in \mathbb R^n$ with growing dimension $n$, 
  we are interested in the convergence behaviour of the Bessel functions $J_{A_{n-1}}(i\lambda(n),\,.\,)$ as  $n\to \infty.$ 
  For this, we consider $J_{A_{n-1}} (\lambda, \,.\,)\,$ as a function on $\mathbb C^r$ for all $r\leq n$ by 
  \begin{equation}\label{short} J_{A_{N-1}}(\lambda, z) := J_{A_{n-1}}\bigl(\lambda, (z,\underline 0_{n-r})\bigr), \quad z\in \mathbb C^r\,.\end{equation}

For later use, we record the following representation. 
  
  \begin{proposition}\label{Bessel_A_rep}
  For $\lambda\in \mathbb C^n$ and $z\in \mathbb C^r$ with $r\leq n$, 
  $$ J_{A_{n-1}}(\lambda	,z)    = \sum_{\kappa \in \Lambda_r^+}  \frac{C_\kappa(\lambda) [kr]_\kappa}{[kn]_\kappa |\kappa|!}\, \mathcal P_\kappa(z),$$
  with the  renormalized Jack polynomials
  $$ \mathcal P_\kappa (z) = \frac{C_\kappa(z)}{C_\kappa(\underline 1_r)} $$
  and the generalized Pochhammer symbol
  $$ [\mu]_\kappa = \prod_{j=1}^{l(\kappa)} (\mu-k(j-1))_{\kappa_j}\quad (\mu\in \mathbb C).$$
  
  \end{proposition}
  
  \begin{proof}
  Consider formula \eqref{Bessel_A}. From \cite[formula (17)]{Ka93}	it is known that for all $\kappa \in \Lambda_r^+$, 
  $$\frac{C_\kappa(\underline 1_r)}{C_{\kappa}(\underline 1_n)} \,=\, \frac{[kr]_\kappa}{[kn]_\kappa}\,.$$
  Together with  the stability property \eqref{stability}, the assertion follows. 
  \end{proof}

  We shall prove the following theorem:
  
\begin{theorem}\label{main} Let $(\lambda(n))_{n\in \mathbb N}$ be a sequence of spectral parameters $\lambda(n) \in \mathbb R^n$ such that each $\lambda(n)$ is decreasing with respect to $\ll.$ Then for fixed multiplicity $k>0,$ the following statements are equivalent. 
\begin{enumerate} \itemsep=2pt

\item[{\rm(1)}] $(\lambda(n))_{n\in \mathbb N}$ is a Vershik-Kerov sequence. 
\item[{\rm(2)}] The sequence of Bessel functions $\bigl(J_{A_{n-1}}(i\lambda(n), \,.\,)\bigr)_{n\in \mathbb N}$ converges uniformly on compact subsets of $\mathbb R^{(\infty)}$, i.e. the convergence is locally uniform on each of the spaces $\mathbb R^r$, $r\in \mathbb N.$ 
\item[{\rm(3)}] The sequence of Bessel functions $\bigl(J_{A_{n-1}}(i\lambda(n), \,.\,)\bigr)_{n\in \mathbb N}$ converges pointwise on $\mathbb R$ to a function which is continuous at $0.$ 
\item[{\rm(4)}] For each fixed multi-index of length $r$, the corresponding coefficients in the Taylor of expansion of $J_{A_{n-1}}(i\lambda(n), \,.\,)\,$ around $0\in \mathbb R^r$ converge as $n \to \infty.$ 
\item[{\rm(5)}] For all symmetric functions $f: \mathbb R^{(\infty)} \to \mathbb C,$ the limit $$ \lim_{n\to\infty} \frac{f(\lambda(n))}{n^{\text{deg} f}}$$ exists. 
\end{enumerate}
Moreover, suppose in this case that $(\alpha, \beta, \gamma)$ are the VK parameters of the VK sequence $(\lambda(n))_{n\in \mathbb N}.\,$
Then 
\begin{equation} \label{limit_A}\lim_{n\to \infty} J_{A_{n-1}}(i\lambda(n),z) \,=\, \prod_{j=1}^\infty e^{i\beta z_j - \frac{\gamma}{2k}z_j^2} \prod_{l=1}^\infty \frac{e^{-i\alpha_l z_j}}{\bigl(1- \frac{i\alpha_l z_j}{k}\bigr)^k}\,,
\end{equation}
where for each $r\in \mathbb N$, the convergence is locally uniform
 in the domain 
\begin{equation}\label{S_r,k} S_{r,k}^\alpha:= \Big\{ z\in \mathbb{C}^{r} : \|\text{Im}\,z\|_\infty <\frac{k}{r|\alpha_1|} \Big\},\end{equation}
 with the notation 
 $\|z\|_\infty = max_{1\leq j\leq n} |z_j|$ for $z\in \mathbb C^r.$
\end{theorem}

\begin{remarks}\label{remBouali} \begin{itemize}\item[{(1)}]	
In the geometric case $k=1$ (i.e. for Hermitian matrices over $\mathbb C$) and for real arguments $z$,
this result essentially goes back to \cite{OV96},  while in \cite{Bo07}, where also $\mathbb F= \mathbb R$ and $\mathbb H$ are considered, only the limit \eqref{limit_A} is established, by completely different methods and under the additional condition $\gamma=0.$ 
\item[{(2)}] We shall prove Theorem \ref{main} by methods which are  inspired by those of \cite{OV96, OO98} and \cite[Chapter 3]{Fa08}.
 They are based on the expansion of the Dunkl-type Bessel functions in terms of Jack  polynomials. 
 The equivalence of parts (1), (2) and (3) of Theorem \ref{main} as well as the convergence statement on $S_{r,k}^\alpha$ were
  already proven 
 in \cite{AN21}  by a different, probabilistic approach. We feel that our method is very natural, which is also suggested by \cite[Remark 1.15]{AN21}.
 \end{itemize}
 \end{remarks}

The proof of Theorem \ref{main} will be carried out  in several steps. We start with the following observation, which is already noted in \cite[Prop. 2.3]{AN21}. For the sake of completeness, we include a short proof which slightly differs from that of \cite{AN21}.

\begin{theorem} \label{VK_Theorem}
	Assume that $(\lambda(n))_{n\in \mathbb N}$ is a VK sequence with  parameters $\omega = (\alpha, \beta,\gamma).$ Then 
	$$ \lim_{n\to \infty} \frac{p_m(\lambda(n))}{n^m} \, = \widetilde p_m(\omega) := \begin{cases}
  \, 1, & m = 0, \\
  \, \beta, & m=1,\\
  \, \delta = \gamma + \sum_{i=1}^\infty \alpha_i^2\,, & m =2, \\
  \, \sum_{i=1}^\infty \alpha_i^m \,, & m \geq 3,
  	\end{cases}$$
  	where the series in the last case is absolutely convergent. In particular, for each symmetric  function $f$ on $\mathbb R^{(\infty)}, $ the limit 
  	$$ \widetilde f(\omega) := \lim_{n\to\infty}\frac{f(\lambda(n))}{n^{\text{deg} f}}$$
  	exists. 
\end{theorem}

\begin{proof} We only have to consider the case $m \geq 3.$ In view of the ordering of $\lambda(n)$ we have for arbitrary $N \in \mathbb N$ that 
\begin{equation}\label{sum_ab} \sum_{i=N}^\infty  \Big\vert \frac{\lambda(n)_i}{n}\Big\vert^m \,\leq \, \Big\vert \frac{\lambda(n)_N}{n}\Big\vert^{m-2} \cdot \frac{p_2(\lambda(n))}{n^2}.\end{equation}
The expression on the right side converges to $\,\alpha_N^{m-2}\delta$ as $n\to \infty.$ 
As $\alpha$ is square-summable by Lemma \ref{VK_Lemma}, this implies that for each $\epsilon >0,$ there exists an index $N\in \mathbb N$ such that for all $n \in \mathbb N,$
\begin{equation}\label{9.3} \sum_{i=N}^\infty |\alpha_i|^m \,+ \,\sum_{i=N}^\infty \Big\vert \frac{\lambda(n)_i}{n}\Big\vert^m \,< \epsilon.\end{equation}
Estimate \eqref{9.3} further leads to 
\begin{align*}
\Big\vert \frac{p_m(\lambda(n))}{n^m} - p_m(\alpha)\Big\vert \leq &\, \sum_{i=N}^\infty |\alpha_i|^m \, +\,\sum_{i=N}^\infty \Big\vert \frac{\lambda(n)_i}{n}\Big\vert^m\, + \sum_{i=1}^{N-1} \Big\vert \frac{\lambda(n)_i^m}{n^m} - \alpha_i^m\Big\vert \\
\leq & \, \epsilon +  \sum_{i=1}^{N-1} \Big\vert \frac{\lambda(n)_i^m}{n^m} - \alpha_i^m\Big\vert.\end{align*}
By the definition of a VK sequence, the last sum tends to zero as $n \to 
\infty.$ As $\epsilon>0$ was arbitrary, 	this finishes the proof.
\end{proof}

 We next consider for  $\lambda \in \mathbb C^{(\infty)} $
 the complex function
 $$ \Phi(\lambda;z):= \prod_{j=1}^\infty \frac{1}{(1-\lambda_jz)^k}\,,$$
 where $\zeta \mapsto \zeta^k$ denotes the principal holomorphic branch of the power function on 
$\mathbb C \setminus ]-\infty,0].$ For fixed $\lambda$, 
 the product is finite and $\Phi(\lambda;\,.\,)$ is holomorphic in a neighborhood  of $0$ in $\mathbb C.$
 According to formula (2.9) of \cite{OO98}, 
 $$ \Phi(\lambda;z) = \, \sum_{j=0}^\infty g_j(\lambda)z^j $$
 with 
 \begin{equation}\label{g_j} g_j(\lambda) = \sum_{i_1 \leq \ldots\leq i_j}\!\! \frac{(k)_{m_1}(k)_{m_2}\cdots}{m_1!\, m_2!\cdots } \, \cdot \lambda_{i_1}\cdots \lambda_{i_j},\end{equation}
 where  $\, m_l := \# \{r\in \mathbb N:i_r = l\} $ denotes the multiplicity of the number $l$ in the tuple $(i_1, \ldots, i_j)$  and $(k)_m = k(k+1) \cdots (k+m-1)$ is the Pochhammer symbol.

 Moreover, from \cite[formula (2.8)]{OO98} and the connection between the $C$- and $P$-normalizations of the Jack polynomials according to formula (12.135) of \cite{Fo10},  
 one calculates that
 \begin{equation}\label{g_C} g_j(\lambda) = \frac{(k)_j}{j!} \cdot   C_{(j)}(\lambda).\end{equation}
 (For partitions $\kappa = (j)$ with just one part, the Jack polynomials $C_{(j)}$ and $P_{(j)}$ coincide). 
 
  \begin{lemma}\label{unique} Suppose $\omega = (\alpha, \beta, \gamma)$ are the VK parameters of a Vershik-Kerov sequence. Then the following hold. 
  \begin{enumerate}
  \item[\rm{(1)}] The infinite product
 $$ \Psi(\omega;z)	:= e^{k\beta z+ \frac{k\gamma}{2}z^2}\prod_{l=1}^\infty \frac{e^{-k\alpha_l z}}{(1-\alpha_l z)^k}  $$
  is holomorphic in $\,S:= \mathbb C \setminus \bigl( \,\big ]-\infty, -\frac{1}{|\alpha_1|}\big]\cup \big[\frac{1}{|\alpha_1|}, \infty\big[ \,\bigr)\,.$
  If $\alpha_l\geq 0$ for all $l\in \mathbb N,$ then
  $\Psi(\omega;\,.\,)  $ is holomorphic in $\, \widetilde S:= \mathbb C \setminus  \big[\frac{1}{\alpha_1}, \infty\big[ \,.$  
  \item[\rm{(2)}] $\omega$ is uniquely determined by $\Psi(\omega;\,. \,).$  
  \end{enumerate}
  
  \end{lemma}
  
 \begin{proof}
(1) Power series expansion around $z=0$ shows that for $|\alpha_lz|\leq \delta <1,$ 
$$ \Big\vert 1 - \frac{e^{-k\alpha_lz}}{(1-\alpha_lz)^k}\Big\vert \leq C_\delta |\alpha_lz|^2 $$
with some constant $C_\delta>0.$ Recall that $\alpha$ is decreasing w.r.t. $\ll$ and square-summable. Hence for fixed $n \in \mathbb N,$ the product 
$$ \prod_{l=n}^\infty  \frac{e^{-k\alpha_lz}}{(1-\alpha_lz)^k}$$
defines a holomorphic function in the open disc 
$\,\bigl\{z\in \mathbb C: |z| < 1/|\alpha_n|\in [0, \infty]\bigr\}.$ 
Moreover, 
$$ \prod_{l=1}^{n-1}  \frac{e^{-k\alpha_lz}}{(1-\alpha_lz)^k}$$
is holomorphic in $S$ and even in $\widetilde S$ if $\alpha_l \geq 0$ for all $l.$ As $\lim_{l\to \infty} \alpha_l = 0,$ it follows that  
$\psi(\omega; \,.\,)$ is holomorphic in $S$ or even in $\widetilde S.$
Unless $\alpha$ is identical zero (which is equivalent to $\alpha_1 =0$), $\Psi(\omega;\,.\,)$ has a singularity in $z= \frac{1}{\alpha_1}.$

(2)
 If $\psi(\omega;\,.)$ is entire, then $\alpha_1 =0.$ Otherwise   $\lim_{z \to  1/\alpha_1} |\Psi(\omega;z)| =\infty. $  Thus $\alpha_1$ is uniquely determined by $\Psi(\omega;\,.\,).$ Multiplying  successively by  $(1-\alpha_1 z)^k,  \ldots ,$ we further obtain that $\alpha_2, \alpha_3, \ldots$ are uniquely determined by $\Psi(\omega; \,.\,)$ as well. It is then obvious that also $\beta$ and $\gamma$ are uniquely determined by $\Psi(\omega;\,.\,).$  
\
 \end{proof}

\begin{proposition}\label{9.6} (1) For $\lambda \in \mathbb C^{(\infty)}$ 
and $z \in \mathbb C$ with $|z| < 1/\max_{j\in \mathbb N} |\lambda_j|$,
\begin{equation}\label{9.5} \Phi(\lambda;z) = \exp\Bigl(k\sum_{m=1}^\infty p_m(\lambda) \frac{z^m}{m}\Bigr).\end{equation}
(2) Moreover, if $(\lambda(n))_{n\in \mathbb N}$ is a VK sequence with parameters $\omega = (\alpha, \beta, \gamma),$ then 
$$ \lim_{n\to\infty} \Phi\Bigl(\frac{\lambda(n)}{n}; z\Bigr) =  \Psi(\omega;z),$$
where the convergence is locally uniform in $\,\{z \in \mathbb C : |z|<1/|\alpha_1|\}$.  
	\end{proposition}

\begin{proof} (1) The left hand side of \eqref{9.5} is  holomorphic on the domain $\{z \in \mathbb{C}:|z|<1/ \max_j |\lambda_j|\}$. As $|p_m(\lambda)| \leq r \max_j|\lambda_j|^m$ for $\lambda \in \mathbb C^r$,  the right hand side is holomorphic on the same domain. Since both sides of \eqref{9.5} have value $1$ in $z=0$, it suffices to verify that they have the same logarithmic derivative. 
Let $\text{log}$ be the principle  holomorphic branch of the logarithm in $\mathbb C \setminus ]-\infty,0].$ Then for $|z|$ small enough, 
$$ \frac{d}{dz} \text{log}\, \Phi(\lambda;z) = \,  \sum_{j=0}^\infty \frac{k\lambda_j}{1-\lambda_j z} \, =\, k \sum_{m=0}^\infty p_{m+1}(\lambda)z^m\,.$$
This is exactly the logarithmic derivative of the right-hand side in \eqref{9.5}.
\smallskip

(2)  For the 	second  assertion, recall from \eqref{sum_ab} that for $m\geq 2$ we may estimate
\begin{equation*} \Big\vert p_m\Bigl(\frac{\lambda(n)}{n}\Bigr)\Big\vert  \leq \Big\vert\frac{\lambda(n)_1}{n}\Big\vert^{m-2} \cdot \frac{p_2(\lambda(n))}{n^2}.\end{equation*} 
Since the right-hand side converges to $|\alpha_1|^{m-2}\delta$ as $n\to\infty,$ the sequence on the left-hand side is uniformly bounded in $n.$ Moreover, for each $\epsilon > 0$ there exists some index $N_\epsilon$ such that
$$ \Big\vert p_m\Bigl(\frac{\lambda(n)}{n}\Bigr)\Big\vert  \leq (|\alpha_1|+\epsilon)^{m-2}\cdot 2 \delta \quad \text{for all } n \geq N_\epsilon\,.$$ Hence for $n \geq N_\epsilon$, the series 
$$ h_n(z) =\sum_{m=1}^\infty p_m\Bigl(\frac{\lambda(n)}{n}\Bigr) \frac{z^m}{m}$$
converges for $|z|<(|\alpha_1|+\epsilon)^{-1}$, and the dominated convergence theorem shows that 
 $$ \lim_{n\to \infty } h_n(z) = \,\sum_{m=1}^\infty \widetilde p_m(\omega) \frac{z^m}{m}\,$$ 
 locally uniformly in $\{|z|<(|\alpha_1|+\epsilon)^{-1}\}$. Thus
\begin{equation}\label{expk} \lim_{n\to\infty} \Phi\Bigl(\frac{\lambda(n)}{n};z\Bigr) \,=\, \exp \Bigl( k \sum_{m=1}^\infty \widetilde p_m(\omega) \frac{z^m}{m}\Bigr)\end{equation}
locally uniformly in the disc $\{ z \in \mathbb C: |z|<1/|\alpha_1|\}.$ 
Now consider $\Psi(\omega;\,.\,)$, which is holomorphic in this disc.
Taking the logarithmic derivative  as in the proof of \cite[Prop. 3.12]{Fa08} and recalling Theorem \ref{VK_Theorem}, we obtain
$$ \frac{d}{dz} \log \Psi(\omega;z) = k\Bigl[\beta + \gamma z -\sum_{l=1}^\infty \bigl(\alpha_l - \frac{\alpha_l}{1-\alpha_l z}\Bigr)\Bigr] = \,k\!\sum_{m=0}^\infty 
\widetilde p_{m+1}(\omega)z^m.$$
The right-hand side in equation \eqref{expk} has the same logarithmic derivative. Since $\Phi\bigl(\frac{\lambda(n)}{n}; 0 \bigr) = 1 = \Psi(\omega;0),$ this proves the stated limit. 
\end{proof}

We now consider the asymptotic  behaviour of the Bessel functions $J_{A_{n-1}}$ as $n\to \infty.$ 
For $z\in \mathbb C^{(\infty)}$, we put
$$ \widehat \Psi(\omega; z) := \prod_{j=1}^\infty \Psi(\omega; z_j),$$
which is actually a finite product.

\begin{theorem}\label{Bessel_A_limit}
Assume that $(\lambda(n))_{n\in \mathbb N}$ is a VK sequence 	with parameters $\,\omega=(\alpha, \beta,\gamma).$ For  $r\in \mathbb N$ and $k>0$ consider the domain $S_{r,k}^\alpha\subset \mathbb C^r$ as defined in Theorem \ref{main}. 
Then the Bessel functions of type $A_{n-1}$ with multiplicity $k$ satisfy 
\begin{equation}\label{main_limit} \lim_{n\to\infty} J_{A_{n-1}}(i\lambda(n), z) = \,\widehat\Psi\Bigl(\omega;\frac{iz}{k}\Bigr) = \prod_{j=1}^\infty e^{i\beta z_j-\frac{\gamma}{2k}z_j^2} \prod_{l=1}^\infty \frac{e^{-i\alpha_lx_j}}{\bigl(1 - \frac{i\alpha_l}{k}z_j\bigr)^k}\,\end{equation}
locally uniformly on $ S_{r,k}^\alpha $.
\end{theorem}

\begin{remark} This theorem is already proven in \cite[Proposition 6.7]{AN21}. Our approach via Jack polynomial expansions allows a  shorter proof, which is given below. It does in particular not require the product formula for Bessel functions used in \cite{AN21}. 
\end{remark}

 \begin{proof} (1) As in \cite[Prop.6.8]{AN21} we first prove that the family $\,\bigl(J_{A_{n-1}}(i\lambda(n), \,.\,)\bigr)_{n\in \mathbb N}$ is uniformly bounded on compact subsets of 
$S_{r,k}^\alpha.$ For this,  we start with a rank-one reduction as in \cite{AN21}. Assume that $n\geq r$ and recall from representation \eqref{int_rep} that there exists a compactly supported probability measure $\mu_n$ on $\mathbb R^n$ such that for all $z\in \mathbb C^r,$
 \begin{equation*} J_{A_{n-1}}(i\lambda(n),z) = \int_{\mathbb R^n} e^{i\langle\xi,(z, \underline 0_{n-r})\rangle} d\mu_n(\xi)
 \end{equation*}	
Hence
 \begin{align*} |J_{A_{n-1}}(i\lambda(n),z)| &\,\leq \, \int_{\mathbb R^n} e^{-\sum_{j=1}^r \xi_j \text{Im}\, z_j} d\mu_n(\xi) \\
 & \,\leq \,\prod_{j=1}^r \Bigl(\int_{\mathbb R^n} e^{-r\xi_j \cdot \text{Im}\,z_j}d\mu_n(\xi)\Bigr)^{\frac{1}{r}} \,=\, \prod_{j=1}^r J_{A_{n-1}}(\lambda(n), -r\text{Im}\, z_j)^{\frac{1}{r}},
 \end{align*}
 where H\"older's inequality was used for the second inequality.  
 In order to prove the claimed boundedness property, it therefore 
 suffices to show that $\bigl(J_{A_{n-1}}(\lambda(n), \,.\,)\bigr)_{n\in \mathbb N}$ is locally uniformly bounded on the interval $ \{ x\in \mathbb R: |x| < \frac{k}{|\alpha_1|}\}.$
 
 For this, we employ the Jack expansion of $J_{A_{n-1}}.$ Recall that in  rank one, the Jack polynomials are the monomials $\mathcal P_{(j)}(x) = x^j,\, j\in \mathbb N_0$.  So from Proposition \ref{Bessel_A_rep}  we obtain for $x\in \mathbb R$ the estimate
\begin{align} \label{estimate_A}
J_{A_{n-1}}(\lambda(n),x) \,& = \, \sum_{j=0}^\infty \frac{C_{(j)}(\lambda(n)) (k)_j}{(kn)_j \, j !}\, x^j \leq \,   \sum_{j=0}^\infty \frac{|C_{(j)}(\lambda(n))| (k)_j}{(kn)^j \, j !}\, |x|^j \notag \\ 
\, & 	= \, \sum_{j=0}^\infty \frac{\big|C_{(j)}\bigl(\frac{\lambda(n)}{n}\bigr)\big| (k)_j}{j !}\,      \Bigl(\frac{|x|}{k}\Bigr)^j \,=\, \sum_{j=0}^\infty \Big\vert g_j\Bigl(\frac{\lambda(n)}{n}\Bigr)\Big\vert 
\Bigl(\frac{|x|}{k}\Bigr)^j \end{align}
with the coefficients $g_j$ from formula \eqref{g_C}.  
 Proposition \ref{9.6} implies that for $n\to \infty$, 
 \begin{equation}\label{psilimit} \sum_{j=0}^\infty g_j\Bigl(\frac{\lambda(n)}{n}\Bigr) z^j \, = \, \Phi\Bigl( \frac{\lambda(n)}{n};z\Bigr) \,\longrightarrow \,\Psi(\omega; z),\end{equation}
  where the convergence is locally uniform in $\{z\in \mathbb C: |z| < 1/|\alpha_1|\}.$
Fix $\rho>0$ with  $\rho < 1/|\alpha_1|.$  Cauchy's inequalities for holomorphic functions and the existence of the limit in \eqref{psilimit} then imply that  
for sufficiently large $n$, 
$$ \Big\vert g_j\Bigl(\frac{\lambda(n)}{n}\Bigr)\Big\vert \, \leq \frac{1}{\rho^j} \cdot \sup_{|z|= \rho} \Big\vert\Phi\Bigl( \frac{\lambda(n)}{n}; z\Bigr)\Big\vert\, \leq \, \frac{C(\rho)}{\rho^j}  $$
with a constant $C(\rho)>0$ independent of $n$ and $j$. 
Together with estimate \eqref{estimate_A}, this proves the assertion and finishes the first step.

\smallskip 

(2) In a second step, we prove the stated convergence result for the sequence 
$$ \varphi_n(z):= J_{A_{n-1}}(i\lambda(n),z),$$ 
which is locally uniformly bounded on  $S_{r,k}^\alpha $ according to step (1). By Montel's theorem we can therefore find a subsequence 
$(\varphi_{n_j})$ which 
 converges locally uniformly to a holomorphic function $\varphi$ on $S_{r,k}^\alpha$.
In some neighborhood of $0\in \mathbb C^r$, this function has a Jack expansion 
$$ \varphi(z) = \sum_{\kappa \in \Lambda_r^+} a_\kappa \mathcal P_\kappa(z)$$
with certain coefficients $a_k \in \mathbb C.$ By the uniform convergence of $(\varphi_{n_j}),$ the coefficients in the Jack expansion of $\varphi_{n_j}$  must converge (as $j\to \infty$) to the corresponding coefficients of $\varphi$. 
In view of Proposition \ref{Bessel_A_rep}, this means that
 \begin{equation*}\lim_{j\to\infty} \frac{i^{|\kappa|} [kr]_\kappa  C_\kappa((\lambda(n_j)))}{[kn_j]_\kappa\,|\kappa|!} \, =\, a_\kappa.\end{equation*}
 But as $\,[kn]_\kappa \sim (kn)^{|\kappa|}$ for $n\to \infty$, we obtain from Theorem \ref{VK_Theorem} that
 $$ a_\kappa = \frac{i^{|\kappa|} [kr]_\kappa  \widetilde C_\kappa(\omega)}{k^{|\kappa|}|\kappa|!} \,.$$
  The Cauchy identity for Jack polynomials, see for instance \cite[Prop. 2.1]{St89}, states for $\lambda \in \mathbb C^{(\infty)}$ and $z\in \mathbb C^r$  with $|z_j|$ small enough  that
 $$ \sum_{\kappa\in \Lambda_r^+} \frac{[kr]_\kappa}{|\kappa|!} C_\kappa(\lambda) \mathcal P_\kappa(z) \, =\, \prod_{j,l} \frac{1}{(1-\lambda_lz_j)^k}\, = \, \prod_{j=1}^r \Phi (\lambda;z_j).$$
 Thus by Proposition \ref{9.6}, 
 $$ \sum_{\kappa\in \Lambda_r^+} \frac{i^{|\kappa|}[kr]_\kappa C_\kappa\bigl(\frac{\lambda(n)}{n}\bigr)}{|\kappa|!\,k^{|\kappa|}} \,\mathcal P_\kappa(z) \, \longrightarrow \, \widehat\Psi\bigl(\omega; \frac{iz}{k}\bigr)	\quad \text{ for $n\to\infty$,}$$
 where the convergence is locally uniform in $z$ in some open neighborhood of $0\in \mathbb C^r.$
 As $\lim_{n\to\infty} C_\kappa\bigl(\frac{\lambda(n)}{n}\bigr) = \widetilde C_\kappa(\omega),$ we conclude that
 $$ \varphi(z) = \widehat\Psi\bigl(\omega; \frac{iz}{k}\bigr)  \quad\text{ for all } z\in S_{r,k}^\alpha.$$ 
 Finally, using Montel's theorem again, we obtain for the full sequence $(\varphi_n)$ that $\varphi_n(z) \, \longrightarrow \, \varphi(z) = \widehat\Psi\bigl(\omega; \frac{iz}{k}\bigr) $ as $n\to \infty,$ locally uniformly in $z\in S_{r,k}^\alpha$. This finishes the proof of the theorem.  
 \end{proof}

\begin{remark}\label{rem_Psi} The above proof shows that for $z \in \mathbb C^r$ with $\|z\|_\infty < \frac{1}{|\alpha_1|}, $ 
$$ \widehat \Psi(\omega;z) = \sum_{\kappa \in \Lambda_r^+}  \frac{[kr]_\kappa }{|\kappa|!}\widetilde C_\kappa (\omega)\mathcal P_\kappa(z).$$
\end{remark}

\begin{remark}\label{nonneg} If the VK sequence $(\lambda(n))_{n\in \mathbb N}$ is nonnegative, then the limit result \eqref{main_limit} holds uniformly on compact subsets of the domain
$$ S_k^\alpha:= \Big\{ z\in \mathbb{C}^{(\infty)} : \|\text{Im}\,z\|_\infty <\frac{k}{|\alpha_1|} \Big\}.$$

\smallskip
This is obtained as follows: Start from Proposition \ref{Bessel_A_rep} for $z\in \mathbb C^r$ and note that $\lambda(n) \geq 0$ implies that $C_\kappa(\lambda(n)) \geq 0$ for all $\kappa \in \Lambda_r^+$, due to the non-negativity of the coefficients of the $C_\kappa$ in their monomial expansion, c.f. \cite{KS97}. Further observe 
that 
$$ [kn]_\kappa \geq (k(n-r+1))^{|\kappa|}$$
for $\kappa \in \Lambda_r^+$.
Similar to the proof of Theorem \ref{Bessel_A_limit}, but without rank-one reduction we therefore obtain
\begin{align*} \big\vert J_{A_{n-1}}(i\lambda(n),z)\big\vert \, & \leq \, \sum_{\kappa \in \Lambda_r^+} \frac{[kr]_\kappa}{|\kappa|!} \, C_\kappa \Bigl(\frac{\lambda(n)}{n-r+1}\Bigr)\cdot \Big\vert \mathcal P_\kappa\Bigl(\frac{\text{Im}\, z}{k} \Bigr)\Big\vert \\
& \leq \, \sum_{\kappa \in \Lambda_r^+} \frac{[kr]_\kappa}{|\kappa|!} \, C_\kappa \Bigl(\frac{\lambda(n)}{n-r+1}\Bigr)\cdot  \mathcal P_\kappa\Bigl(\frac{|\text{Im}\, z|}{k} \Bigr) \\
& = \, \prod_{j=1}^r \Phi\Bigl(\frac{\widetilde \lambda(n)}{n}\,; \frac{|\text{Im}\, z_j|}{k}\Bigr),\end{align*}
where  $\widetilde 
\lambda(n) = \frac{n}{n-r+1}\lambda(n)$ and $\,|\text{Im}\,z|$ is understood componentwise. Observe that $(\widetilde\lambda(n))$ is also Vershik-Kerov with the same VK parameters as $(\lambda(n)).$ So for $n\to\infty,$
$$ \prod_{j=1}^r \Phi\Bigl(\frac{\widetilde \lambda(n)}{n}\,; \frac{|\text{Im}\, z_j|}{k}\Bigr) \,\longrightarrow \,  \widehat \Psi\Bigl(\omega; \frac{|\text{Im}\, z|}{k}\Bigr)$$
 locally uniformly in $\,z\in S_k^\alpha \cap \,\mathbb C^r.$ This shows that the family $(J_{A_{n-1}}(i\lambda(n), \,.\,)$ is uniformly bounded on compact subsets of $S_k^\alpha.$ Proceeding as in the proof of Theorem \ref{Bessel_A_limit} then yields the assertion.  
\end{remark}

Let us now continue with the proof of Theorem \ref{main}. The implication $(2) \Rightarrow (1)$ will be established by the following

\begin{lemma}\label{9.11} Consider a sequence $(\lambda(n))_{n\in \mathbb N}$ such that each $\lambda(n)\in \mathbb R^n$ is decreasing with respect to $\ll.$ Suppose that the sequence of Bessel functions $\, J_{A_{n-1}}(i\lambda(n), \,.\,)$ converges pointwise on $\mathbb R$ to a function which is continuous at $0$. Then $(\lambda(n))$ is a VK sequence.
\end{lemma}

\begin{proof} 
Put $\varphi_n(x) := J_{A_{n-1}}(i\lambda(n), x), \, x\in \mathbb R$ and $\varphi(x):= \lim_{n\to\infty}\varphi_n(x).$ Again in view of \eqref{int_rep}, there are
compactly supported probability  measures $\mu_n $ on $\mathbb R$ such that 
$$ \varphi_n(x) = \int_{\mathbb R} e^{ix\xi} d\mu_n(\xi) \quad \text{for all } x\in \mathbb R.$$
By L\'evy's continuity theorem, there exists a probability measure $\mu$ on $\mathbb R$ such that 
$\mu_n \to \mu$ weakly and 
$$ \varphi(x) = \int_{\mathbb R} e^{ix\xi} d\mu(\xi) \quad \text{for all } x\in \mathbb R.$$
In particular, the family of measures $\{\mu_n: n\in \mathbb N\}$ is tight. Recall the functions $g_j(\lambda)$ from \eqref{g_j}. By  Proposition \ref{Bessel_A_rep} and formula \eqref{g_C} we have
$$ \varphi_n(x) = \sum_{j=0}^\infty \frac{C_{(j)}(\lambda(n))\cdot (k)_j}{(kn)_j\cdot j!} (ix)^j\,=\, \sum_{j=0}^\infty \frac{g_j(\lambda(n))}{(kn)_j} (ix)^j\,.$$
Hence the moments of the measures $\mu_n$ are given by
$$ \int_{\mathbb R} \xi^j d\mu_n(\xi) = j!\, \frac{g_j(\lambda(n))}{(kn)_j}.$$
We now employ Lemma 5.2 of \cite{OO98}. 
From the definition of the functions $g_j$ one can find a constant $C>0$ such that $\, g_4(\lambda) \leq  Cg_2(\lambda)^2$ for all $\lambda \in \mathbb R^{(\infty)}, $ which shows that the quotient
$$ \frac{\int_{\mathbb R} \xi^4 d\mu_n(\xi)}{\Bigl(\int_{\mathbb R} \xi^2 d\mu_n(\xi)\Bigr)^2}$$
is bounded as a function of $n\in \mathbb N.$ Hence we conclude from Lemma 5.1. of \cite{OO98} that
the sequence $\bigl(\int_{\mathbb R} \xi^2 d\mu_n(\xi)\bigr)_{n\in \mathbb N}$ is bounded, which in turn implies that the sequence $\bigl(\frac{g_2(\lambda(n))}{n^2}\bigr)_{n\in \mathbb N}$ is bounded. As $\, 2g_2 = k^2p_1^2 + kp_2,$ 
the sequences
\begin{equation}\label{p_1p_2} \Bigl(\frac{|p_1(\lambda(n))|}{n}\Bigr)_{n\in \mathbb N} \quad \text{ and } \, \Bigl(\frac{p_2(\lambda(n))}{n^2}\Bigr)_{n\in \mathbb N}\end{equation}
are bounded as well. 
Standard compactness arguments and a diagonalization argument imply that $(\lambda(n))_{n\in \mathbb N} $ has a subsequence which is Vershik-Kerov. Finally, consider two such subsequences $(\lambda_l(n))_{n\in \mathbb N}$ with VK parameters $\omega_l, \, l=1,2.$ Then by Theorem \ref{Bessel_A_limit} and our assumptions,
$$ \varphi(x) = \lim_{n\to \infty} J_{A_{n-1}}(i\lambda_l(n),x) = \Psi \Bigl(\omega_l; \frac{ix}{k}\Bigr)\quad \text{ for all }x\in \mathbb R.$$
Hence $\Psi(\omega_1;\,.\,) = \Psi(\omega_2;\,.\,),$ and  Proposition \ref{unique} implies that $\omega_1 = \omega_2.$ It follows that the full sequence $(\lambda(n))_{n\in \mathbb N}$ is Vershik-Kerov. 
\end{proof}

Putting things together, we are now able to finalize the proof of Theorem \ref{main}. 

\begin{proof}[Proof of Theorem \ref{main}] 
The implication $(1) \Rightarrow (2)$ is contained in Theorem \ref{Bessel_A_limit}.  
Implication $(2) \Rightarrow (3)$ is trivial, while $(3) \Rightarrow (1)$ is just Lemma \ref{9.11}. Further, Theorem \ref{VK_Theorem} proves the  implication $(1) \Rightarrow (5).$ The equivalence of statements (4) and (5) is obvious from  Proposition \ref{Bessel_A_rep},  because  the Jack polynomials span the algebra of symmetric functions. It thus remains to prove the implication $(5)\Rightarrow (1).$ For this, suppose that  $(\lambda(n))_{n\in \mathbb N} $ is a sequence with each $\lambda(n) \in \mathbb R^n$ decreasing w.r.t $\ll, $  and such that $\,\lim_{n\to \infty} \frac{f(\lambda(n))}{n^{\text{deg} f}}\,$ exists for all symmetric functions $f.$ Then in particular, the sequences  $\bigl(\frac{p_1(\lambda(n))}{n}\bigr)$ and $\bigl(\frac{p_2(\lambda(n))}{n^2}\bigr)$ are bounded. Again by a compactness argument, $(\lambda(n))$ has a subsequence which is Vershik-Kerov. Suppose $(\lambda_l(n)), \, l=1,2$ are two such subsequences with VK parameters $\omega_l\,.$ Then by Theorem \ref{Bessel_A_limit}, the sequences $\bigl(J_{A_{n-1}}(i\lambda_1(n), \,.\,)\bigr)$ and 
$\bigl(J_{A_{n-1}}(i\lambda_2(n), \,.\,)\bigr)$ converge locally uniformly on $\mathbb R^r$ to the same limit,  because for each $\kappa \in \Lambda_r^+$, the limit
$$ \lim_{n\to\infty} \frac{C_\kappa(\lambda_l(n))}{n^{|\kappa|}}\,=\, \lim_{n\to\infty} \frac{C_\kappa(\lambda(n))}{n^{|\kappa|}}$$
is independent of $l.$ Arguing further as in the proof of Lemma \ref{9.11}, we obtain that $\omega_1 = \omega_2$ and that $(\lambda(n))$ is a VK sequence.This finishes the proof of the theorem. 
 \end{proof}

 We shall now parametrize the possible limit functions in Theorem \ref{Bessel_A_limit}. We put
 $$ \Omega:= \Bigl\{(\alpha, \beta,\gamma): \beta\in \mathbb R, \gamma\geq 0, \alpha=(\alpha_i)_{i\in \mathbb N} \text{ with } \alpha_i \in \mathbb R, \alpha_{i+1}\ll \alpha_i\,, \sum_{i=1}^\infty \alpha_i^2 < \infty\Bigr\}.$$ 
 Note that for $(\alpha, \beta, \gamma) \in \Omega,$ either all entries of $\alpha$ are non-zero, or all entries up to finitely many  are zero.

 \begin{proposition}\label{9.12}
 For any element $\omega=(\alpha, \beta, \gamma) \in \Omega$ there exists a VK sequence $(\lambda(n))$ with $VK$ parameters $\omega.$ 
 	\end{proposition}

 \begin{proof}
 We divide the proof into several steps. 
 
 \smallskip
 \noindent
 (i) Assume that $\alpha = 0.$ Then for arbitrary $\epsilon >0,$ there exists a sequence $x=(x_i)_{i\in \mathbb N}$ in $\mathbb R$ such that
 \begin{equation}\label{rearrange} |x_i|\leq \epsilon \,\, \text{ for all }\, i \in \mathbb N, \>\> \sum_{i=1}^\infty x_i = \beta \quad\text{and } \,\, \sum_{i=1}^\infty x_i^2 = \gamma.\end{equation}
 To see this, choose $N \in \mathbb N$ such that $\, \bigl(\frac{6\gamma}{\pi^2 N}\bigr)^{1/2}\leq \epsilon\, $ and start with the alternating sequence
 $$ x_i^\prime := \bigl(\frac{6\gamma}{\pi^2 N}\bigr)^{1/2}\cdot\frac{(-1)^i}{k+1}\quad \, \text{ if } \> k < \frac{i}{N}\leq k+1,\,  k \in \mathbb N_0.\,$$  
 It satisfies the first and the third condition of \eqref{rearrange}, and by the Riemann rearrangement theorem, there exists a rearrangement $(x_i)_{i\in \mathbb N}$
of $(x_i^\prime)_{i\in \mathbb N}$ satisfying the second condition as well. 
For each $m \in \mathbb N$ we can therefore find a real sequence $x^{(m)}= (x_i^{(m)})$ and and index $n_m \in \mathbb N$ with $n_m\to \infty $ for $m \to \infty,$ such that for all $n \geq n_m\,$,
$$ \big\vert  x_i^{(m)} \big\vert \leq \frac{1}{m} \quad \text{for all }\,i\in \mathbb N, \quad \vert \sum_{i=1}^n x_i^{(m)} \, -\beta \vert \,\leq \frac{1}{m}\,, \quad  \vert \sum_{i=1}^n \bigl(x_i^{(m)}\bigr)^2 \, -\gamma \vert\, \leq  \frac{1}{m}\,.$$
We may also assume that $n_{m+1} > n_m$ for all $m$. Rearranging the entries of each tuple $(x_1^{(m)}, \ldots, x_{n_m}^{(m)})$ according to  $\ll,$ we thus obtain a sequence $(\lambda(n_m)^\prime)_{m\in \mathbb N}$ where each $\lambda(n_m)^\prime\in \mathbb R^{n_m}$ is decreasing w.r.t. $\ll$  and satisfies
\begin{align*}
&\lim_{m\to \infty}\lambda(n_m)_i^\prime = 0 \quad \text{ for all } \, i \in \mathbb N,\\	
&\lim_{m\to \infty} \sum_{i=1}^{n_m} \lambda(n_m)_i^\prime \, = \beta, \\
& \lim_{m\to \infty} \sum_{i=1}^{n_m} \bigl(\lambda(n_m)_i^\prime\bigr)^2 \, = \gamma.
\end{align*}
Finally, put $ \, \lambda(n_m):= n_m \lambda(n_m)^\prime $ and $\,\lambda(n):= (n\lambda(n_m)^\prime, 0, \ldots, 0) \in \mathbb R^n $ for $n_m < n < n_{m+1}.$ 
Then $(\lambda(n))_{n \geq n_1} \,$ is a VK sequence with parameters $(\alpha = 0, \beta, \gamma).$ 

\smallskip\noindent
(ii) Assume that $\alpha$ has finitely many non-zero entries and let $m\in \mathbb N$ be maximal such that $\alpha_m \not=0.$ Let $(\lambda(n)^\prime)_{n\in \mathbb N}$ be a VK sequence with parameters $(0, \beta^\prime, \gamma),$ where $\, \beta^\prime = \beta - \sum_{i=1}^m \alpha_i.$ For $n>m,$ put
$$ \lambda(n):= (n\alpha_1, \ldots, n\alpha_m, \lambda(n)_1^\prime, \ldots, \lambda(n)_{n-m}^\prime).$$
For $n$ large enough, say $n \geq n_0$, the entries of $\lambda(n)$ are decreasing with respect to $\ll\,,$  because $\lim_{n\to \infty} \frac{ \lambda(n)_i^\prime}{n} = 0.$ Then $(\lambda(n))_{n\geq n_0}$ is Vershik-Kerov with parameters $(\alpha, \beta, \gamma).$

\smallskip\noindent
(iii) Assume that all entries of $\alpha$ are non-zero. For $m \in \mathbb N, $ put $\,\omega^{(m)}:=(\alpha^{(m)}, \beta, \gamma),$ where $\alpha^{(m)} = (\alpha_1, \ldots, \alpha_m, 0, \ldots \,).$ According to part (ii), there exists a VK sequence 
$(\lambda^{(m)}(n))_{n\in \mathbb N} $ with VK parameters $\omega^{(m)}.$ By a diagonalization argument we obtain a sequence $\, \lambda(n_m) := \lambda^{(m)}(n_m)$ with $n_{m+1} > n_m$ satisfying
\begin{align*} & \lim_{m\to \infty}\lambda(n_m)_i\, = \alpha_i \quad \text{ for all } i\in \mathbb N,\\
&\lim_{m\to \infty} \sum_{i=1}^{n_m} \lambda(n_m)_i \, = \, \beta, \\
& \lim_{m\to \infty} \sum_{i=1}^{n_m} \bigl(\lambda(n_m)_i\bigr)^2  =  \delta = \gamma 
+ \sum_{i=1}^\infty \alpha_i^2. \end{align*}
Finally, for $n \in \mathbb N$ with $n_m \leq n < n_{m+1}$ put $\, \lambda(n):= \bigl( \frac{n}{n_m}\lambda(n_m), 0, \ldots, 0\bigr)\in \mathbb R^n.$ Then $(\lambda(n))_{n\geq n_1}$ is Vershik-Kerov with VK parameters $(\alpha, \beta, \gamma).$   
\end{proof}

Together with Lemma \ref{unique}, this result shows that the possible limits (for $n \to \infty$) of the Bessel functions $J_{A_{n-1}}(i\lambda(n),x)$ with $x \in \mathbb R^r$ and $ \lambda(n) \in \mathbb R^n$ are exactly all the infinite products $ \, \widehat \Psi\bigl(\omega;\frac{ix}{k}\bigr), $ of Theorem \ref{Bessel_A_limit}, 
which are in bijective correspondence with the parameters $\omega\in\Omega.$ 
 
Let us finally come back to the Olshanski spherical pair $(G_\infty, K_\infty)$ 
as in  \eqref{Olshanski_A}. From our results, we obtain  the following corollary. Its first part goes back to  Pickrell \cite{Pi91}, as already mentioned.

\begin{corollary} \label{Cor_A} \begin{enumerate}
\item[\rm{(1)}] The set of positive definite spherical functions of the Olshanski spherical pair
$(G_\infty, K_\infty)=(U_\infty(\mathbb F)\ltimes \text{Herm}_\infty(\mathbb F), U_\infty(\mathbb F)),$ considered  as $U_\infty(\mathbb F)$-invariant functions on $\text{Herm}_\infty(\mathbb F)$, is parametrized by the set $\Omega$ via
$$ \varphi_\omega(X)  = \prod_{j=1}^\infty e^{i\beta x_j - \frac{\gamma}{d} x_j^2} \prod_{l=1}^\infty \frac{e^{-i\alpha_lx_j}}{(1-i\frac{2}{d}\alpha_l x_j)^{d/2}}\,, \quad \omega = (\alpha, \beta, \gamma) \in \Omega,$$
where $(x_1, x_2, \ldots)\in \mathbb R^{(\infty)}$ are the eigenvalues of $X\in \text{Herm}_{\infty}(\mathbb F),$ decreasingly ordered by size and counted according to their multiplicity. 
\item[\rm{(2)}] Consider a sequence  of positive definite spherical functions 
\begin{equation}\label{Olshanski_AA} \varphi_n(X) = J_{A_{n-1}}\bigl(\frac{d}{2}; i\lambda(n), \sigma(X)\bigr), \quad \lambda(n)\in \mathbb R^n	
\end{equation}
 of the Gelfand pairs $(G_n, K_n) = (U_n(\mathbb F)\times \text{Herm}_n(\mathbb F), U_n(\mathbb F)).$ Then $(\varphi_n)$ converges uniformly on compact subsets of $\text{Herm}_\infty(\mathbb F)$ if and only if 
$(\lambda(n))$ is, up to permutation of the entries, a Vershik-Kerov sequence. In 
the  case of convergence, the limit is given by $\varphi_\omega,$ with $\omega$ the VK parameters of $(\lambda(n)).$

	\end{enumerate}

\end{corollary}
	
	\begin{proof} (1)
	For a topological group $H$ 
	consider the set $$ P_1(H) = \{\varphi\in C(H): \varphi \,\text{ positive definite}, \varphi(e)=1\}$$
	and denote by $\text{ex}(P_1(H))$ the set of its extremal points. 
		In   \cite[Theorem 22.10]{Ol90} it is proven that 
	each $\varphi \in \text{ex}(P_1(G_\infty))$ can be approximated uniformly on compact sets
	by a sequence of functions $\varphi_n \in \text{ex}(P_1(G_n)).$	 
	An inspection of the proof shows that this statement remains true for biinvariant functions, i.e. each $K_\infty$-biinvariant $\varphi \in \text{ex}(P_1(G_\infty))$ can be approximated uniformly on compact sets
	by a sequence of $K_n$-biinvariant functions $\varphi_n \in \text{ex}(P_1(G_n)).$  According to  \cite[Theorems 23.3 and 23.6]{Ol90}, the 	positive definite spherical functions of a spherical pair	$(G, K)$ (an Olshanski spherical pair or a Gelfand pair) are exactly those elements of $\text{ex}(P_1(G))$ which are $K$-biinvariant. Thus, for a positive definite spherical function $\varphi$ of $(G_\infty, K_\infty),$ there 
	exists a sequence $(\varphi_n)$  of positive definite spherical functions of $(G_n, K_n)$ which converges uniformly on compact sets to $\varphi.$	(This is also noted in \cite[Theorem 3.5]{OV96}). By Lemma \ref{A_spherical_char}, $\varphi_n$ is given by a positive definite Bessel function $J_{A_{n-1}}$ with multiplicity $k= \frac{d}{2}, \, d= \text{dim}_\mathbb R\mathbb F,$ i.e.
	it is of the form \eqref{Olshanski_AA}. Without loss of generality we may assume that  $\lambda(n)$ is decreasing w.r.t $\ll \,.$  From Theorem \ref{main} it now follows that $(\lambda(n))$ has to be a VK sequence and that $\varphi = \varphi_\omega\,,$ where $\omega$ are the VK parameters of $(\lambda(n)).$ 
	
	Conversely, starting with $\omega\in \Omega$ we may choose an associated VK sequence $(\lambda(n))$ by Proposition \ref{9.12}. 
	Then \eqref{Olshanski_AA}
	defines a sequence $(\varphi_n)$ of positive definite 
	spherical functions of $(G_n, K_n)$ which converge to $\varphi_\omega$ uniformly on compact sets according to Theorem \ref{main}.
	It is then clear from the definitions that $\varphi_\omega$ is a positive definite Olshanski spherical function of $(G_\infty, K_\infty).$ This finishes the proof of Part (1).

	\smallskip
	Part (2) is immediate from Theorem \ref{main}.
	\end{proof}
	
\section{The type $B$ case}\label{type B}

Recall from Section \ref{Bessel}
the Bessel functions $J_{B_n}$ of type $B_n$.  As $n \to \infty$, we shall consider them with the multiplicities $\kappa_n:=(k_n^\prime,k)$ with value $k>0$ on the roots $\pm (e_i\pm e_j) $ and $k_n^\prime \geq 0$ on the roots $\pm e_i.$ It will become clear at the end of this section why the multiplicity parameter $k_n$ is allowed to vary with $n.$ With $$\,\nu_n := k_n^\prime +k(n-1) + \frac{1}{2}$$ we have
\begin{equation}\label{BesselB}
J_{B_n}(\kappa_n;\lambda, z)\, = \, \sum_{\kappa \in \Lambda_n^+} \frac{1}{4^{|\kappa|}[\nu_n]_\kappa}   \frac{C_\kappa(\lambda^2) C_\kappa(z^2)}{|\kappa|!\, C_\kappa(\underline 1_n)},
\end{equation}
where the Jack polynomials are of index $1/k.$ Recall the stability property \eqref{stability} of the Jack polynomials. 
Adopting the notation from \eqref{short}, we therefore have for $\lambda\in \mathbb C^n$ and $z \in \mathbb C^r$ with $r \leq n$ the representation 
\begin{equation}\label{rep_B} J_{B_n}(\kappa_n; \lambda,z) := J_{B_n}(\kappa_n; \lambda,(z, \underline 0_{n-r})= \, \sum_{\kappa \in \Lambda_r^+} 
\frac{C_\kappa(\lambda^2) [kr]_\kappa}{4^{|\kappa|}[kn]_\kappa [\nu_n]_\kappa |\kappa|!}\, \mathcal P_\kappa(z^2).\end{equation}
Recall our notion \eqref{chamber} for the positive Weyl chamber of type $B.$
The following counterpart of Theorem \ref{main} will be the main result of this section.

\begin{theorem}\label{main_B} Consider the Bessel functions $J_{B_n}(\kappa_n; \,.\,,\,.\,)$ with multiplicity $\kappa_n=(k_n^\prime, k),$ where $k>0$ and $k_n^\prime \geq 0$.  Further, 
let $(\lambda(n))_{n\in \mathbb N}$ be a sequence of spectral parameters $\lambda(n) \in \overline C_n$, i.e. the entries of $\lambda(n)$ are non-negative and decreasing. 
 Then the following statements are equivalent. 
\begin{enumerate} \itemsep=2pt

\item[{\rm(1)}]  $\bigl(\frac{\lambda(n)^2}{\nu_n}\bigr)_{n\in \mathbb N}$ is  a Vershik-Kerov sequence.  
\item[{\rm(2)}] The sequence of Bessel functions $\bigl(J_{B_{n}}(\kappa_n;i\lambda(n), \,.\,)\bigr)_{n\in \mathbb N}$ converges locally uniformly on each of the spaces $\mathbb R^r$, $r\in \mathbb N.$ 
\item[{\rm(3)}] The sequence of Bessel functions $\bigl(J_{B_{n}}(\kappa_n;i\lambda(n), \,.\,)\bigr)_{n\in \mathbb N}$ converges pointwise on $\mathbb R$ against  a function which is continuous at $0.$ 
\item[{\rm(4)}] For each fixed multi-index of length $r$, the corresponding coefficients in the Taylor of expansion of $J_{B_{n}}(\kappa_n;i\lambda(n), \,.\,)\,$ around $0\in \mathbb R^r$ converge as $n \to \infty.$ 
\item[{\rm(5)}] For all symmetric functions $f: \mathbb R^{(\infty)} \to \mathbb C,$ the limit $$ \lim_{n\to\infty} \frac{f(\lambda(n)^2)}{(n\nu_n)^{\text{deg} f}}$$ exists. 
\end{enumerate}
In this case, let $\, \omega = (\alpha, \beta, \gamma)$ be the VK parameters of the sequence $\bigl(\frac{\lambda(n)^2}{\nu_n}\bigr).$
Then $\gamma=0, \, \alpha_l \geq 0$ for all $l$, and 
\begin{equation}\label{limit_B} 
\lim_{n\to \infty} J_{B_{n}}(\kappa_n; i\lambda(n),z) \,=\, \widehat \Psi \Bigl(\omega; - \frac{z^2}{4k}\Bigr)\, =\,  \prod_{j=1}^\infty e^{-\frac{\beta z_j^2}{4}} \prod_{l=1}^\infty \frac{e^{\frac{\alpha_l z_j^2}{4}}}{\bigl(1+ \frac{\alpha_l z_j^2}{4k}\bigr)^k}\,,\qquad 
\end{equation} 
where the convergence is  uniform  on compact subsets of the domain
$$ S_{k,+}^\alpha :=  \Big\{ z\in \mathbb{C}^{(\infty)} : \|\text{Im}\,z\|_\infty < \, 2\,\sqrt{ \frac{k}{\alpha_1}} \,\Big\}.$$

\end{theorem}

\begin{remarks}\label{remB} 
1. It is a consequence of Lemma \ref{VK_Lemma}
that the VK parameter $\gamma$ is $0$ in the present situation.

2. We do not have any restrictions on the asymptotic behavior of $\nu_n$ apart from the condition $k_n^\prime \geq 0,$ by which $\nu_n$ grows at least linearly. Only the set of spectral parameters for which the Bessel functions converge depends on $\nu_n$ via condition $(1)$ of the Theorem.  The possible set of limits in \eqref{limit_B} depends only on the VK parameters $\omega$.

3. Assume that $$ \lim_{n\to\infty}\frac{k_n^\prime}{n} = C\,$$ with some constant $C \geq 0.$  Then $\, \nu_n \sim (C+k)n\, $ for $n \to \infty. $ In particular, $\bigl(\frac{\lambda(n)^2}{\nu_n}\bigr)$	is Vershik-Kerov with VK parameters $(\alpha, \beta, 0)$ if and only if $\bigl(\frac{\lambda(n)^2}{n}\bigr)$	is Vershik-Kerov with VK parameters $\bigl((C+k)\alpha, (C+k)\beta, 0\bigr).$
\end{remarks}

For the proof of Theorem \ref{main_B}, we start with the following

\begin{lemma}\label{9.17} Let $(\lambda(n))_{n\in \mathbb N}$ with $\lambda(n) \in \overline C_n$ be such that the sequence $\bigl(\frac{\lambda(n)^2}{\nu_n}\bigr)$	is Vershik-Kerov  with VK parameters $\omega=(\alpha, \beta, 0).$	 
Then for $z \in S_{k,+}^\alpha$ we have
$$  \lim_{n\to\infty} J_{B_{n}}(\kappa_n;i\lambda(n), z) = \,\widehat\Psi\Bigl(\omega;-\frac{z^2}{4k}\Bigr). $$
The convergence is uniform on compact subsets of $S_{k,+}^\alpha.$ 

\end{lemma}

\begin{proof}  Fix $r \in \mathbb N.$ For  $z \in S_{k,+}^\alpha \cap \mathbb C^r$ and $n \geq r$  consider the  functions 
\begin{equation}\label{phi_n_B} \varphi_n(z):= J_{B_n}(\kappa_n; i\lambda(n), z)\, = \, \sum_{\kappa \in \Lambda_r^+} 
\frac{[kr]_\kappa}{|\kappa|!} \cdot  \frac{C_\kappa(\lambda(n)^2)}{[kn]_\kappa [\nu_n]_\kappa}\,\mathcal P_\kappa\Bigl(\frac{-z^2}{4}\Bigr).
\end{equation} 
 
 By \eqref{J_estimate}, the non-negativity of the monomial coefficients of the Jack polynomials and the non-negativity of $\lambda(n)$, we estimate as in Remark \ref{nonneg}:
 \begin{align}\label{phi_n_B_abs} |\varphi_n(z)| \, &\leq \,\varphi_n(-\text{Im}\, z)\notag  \\ 
& \leq \, \sum_{\kappa \in \Lambda_r^+} \frac{ [kr]_\kappa}{|\kappa|!}\,
C_\kappa\Bigl( \frac{\lambda(n)^2}{(n-r+1)(\nu_n-k(r-1))}\Bigr)\mathcal P_\kappa \Bigl( 
\frac{(\text{Im}\, z)^2}{4k}\Bigr) \notag \\
& = \, \prod_{j=1}^r \Phi \Bigl(\frac{\widetilde \lambda(n)}{n}; \frac{(\text{Im}\, z)^2}{4k}\Bigr)
\end{align}
with $$ \widetilde \lambda(n) := \frac{n\nu_n}{(n-r+1)(\nu_n-k(r-1))}\cdot \frac{\lambda(n)^2}{\nu_n}.$$ 
By our assumption on $\lambda(n)$, the sequence $(\widetilde \lambda(n))$ is Vershik-Kerov with parameters $\omega = (\alpha, \beta, 0).$ Hence by Proposition \ref{9.6},  the product on the right side of \eqref{phi_n_B_abs} converges for $n \to \infty$ to 
$\,\widehat \Psi\bigl(\omega; \frac{(\text{Im}\, z)^2}{4k}\bigr),$  uniformly on compact subsets of $  \in S_{k,+}^\alpha.$ In particular, the sequence $(\varphi_n)$ is uniformly bounded on compact subsets of $S_{k,+}^\alpha$. Now consider the coefficients in the expansion of $\varphi_n$. Theorem \ref{VK_Theorem} and the asymptotics $\, [kn]_\kappa[\nu_n]_\kappa \sim (kn\nu_n)^\kappa\, $ yield
$$ \lim_{n\to \infty} \frac{C_\kappa(\lambda(n)^2)}{[kn]_\kappa[\nu_n]_\kappa} \, = \, \frac{\widetilde C_\kappa (\omega)}{k^{|\kappa|}}\,.$$
By a Montel argument as in the proof of Theorem \ref{Bessel_A_limit} and with Remark 
\ref{rem_Psi} in mind, we thus obtain that
$$ \lim_{n\to \infty} \varphi_n(z) \, = \, \widehat \Psi \Bigl(\omega; - \frac{z^2}{4k}\Bigr)$$
locally uniformly on $ S_{k,+}^\alpha \cap \mathbb C^r.$ 
\end{proof}

\begin{lemma}\label{9.18} Consider a sequence $(\lambda(n))_{n\in \mathbb N}$ with $\lambda(n)\in \overline C_n\,.$   Assume that the sequence of Bessel functions $\, J_{B_n}(i\lambda(n), \,.\,)$ converges pointwise on $\mathbb R$ to a function which is continuous at $0$. Then the sequence $\bigl(\frac{\lambda(n)^2}{\nu_n}\bigr)$ is  Vershik-Kerov.
\end{lemma}

\begin{proof} The proof is similar to that of Lemma \ref{9.11}.
For $x\in \mathbb R,$ put 
$$\varphi_n(x) := J_{B_{n}}(\kappa_n;i\lambda(n), x) \, = \, \int_{\mathbb R} e^{ix\xi} d\mu_n(\xi)$$
with certain compactly supported probability measures $\mu_n$ on $\mathbb R$. By the symmetry properties of $J_{B_n}$, the measure $\mu_n$ is even, hence its odd moments vanish. 
Let further $\varphi(x):= \lim_{n\to\infty}\varphi_n(x).$ 
Again by L\'evy's continuity theorem, there exists a probability measure $\mu$ on $\mathbb R$ such that 
$\mu_n \to \mu$ weakly and 
$$ \varphi(x) = \int_{\mathbb R} e^{ix\xi} d\mu(\xi) \quad \text{for all } x\in \mathbb R.$$
Further, the family  $\{\mu_n: n\in \mathbb N\}$ is tight. From  \eqref{rep_B} and formula \eqref{g_C}
we deduce that
$$ \varphi_n(x) = \sum_{j=0}^\infty \frac{g_j\bigl(\lambda(n)^2\bigr)}{4^j (\nu_n)_j (kn)_j}\, (-x)^{2j}.$$   
 This shows that the even moments of $\mu_n$ are given by
$$ \int_{\mathbb R} \xi^{2j} d\mu_n(\xi) = (2j)!\, \frac{g_j(\lambda(n)^2)}{4^{j}(kn)_j(\nu_n)_j}.$$   
As in the proof of Lemma \ref{9.11}, we deduce that  the quotient
$$ \frac{\int_{\mathbb R} \xi^8 d\mu_n(\xi)}{\bigl(\int_{\mathbb R} \xi^4 d\mu_n(\xi)\bigr)^2}$$
is bounded in $n\in \mathbb N.$ Now we conclude from  \cite[Lemma 5.2]{OO98} (employing the Lemma for the image measure of $\mu_n$ under $\xi \mapsto \xi^2$) that
the sequence $\,\bigl(\int_{\mathbb R} \xi^4 d\mu_n(\xi)\bigr)$ is bounded. As $(\nu_n)_2 \sim \nu_n^2 $ and $(kn)_2 \sim (kn)^2$ for $n \to \infty$, it follows that the sequence 
$$\Bigl(g_2\Bigl(\frac{\lambda(n)^2}{n\nu_n}\Bigr)\Bigr)_{n\in \mathbb N}$$ 
is bounded as well. Continuing as in the proof of Lemma \ref{9.11} we obtain that $\bigl(\frac{\lambda(n)^2}{\nu_n}\bigr)$ is a Vershik-Kerov sequence. 
\end{proof}

\begin{proof}[Proof of Theorem \ref{main_B}]
From Theorem  \ref{main} it is clear that the statements (1) and (5) are equivalent. The equivalence of (4) and (5)  follows from expansion \eqref{phi_n_B} and the fact that the Jack polynomials span the algebra of symmetric functions. By Lemma \ref{9.18}, statement (3) implies (1). Finally, Lemma \ref{9.17} shows that statement  (1) implies statement (2), which in turn implies (3).
\end{proof}

We finally want to determine the set of all   parameters $\omega= (\alpha, \beta,0)$ which occur as VK parameters of a non-negative Vershik-Kerov sequence as in Theorem \ref{main_B}. Recall that in the non-negative case, the parameter $\gamma$ is automatically zero due to Lemma \ref{VK_Lemma}.

\begin{proposition} \label{Omega_plus}
The set $\Omega_+$ of all pairs $(\alpha, \beta)$ for which there exists a non-negative VK sequence with parameters $(\alpha, \beta, 0)$ is given by 	
$$ \Omega_+ = \{(\alpha, \beta): \beta \geq 0, \, \alpha = (\alpha_i)_{i\in \mathbb N} \text{ with } \,\alpha_i \in \mathbb R,\,  \alpha_1 \geq \alpha_2 \geq \ldots \geq 0, \, \sum_{i=1}^\infty \alpha_i \leq \beta\,\}.$$ 
\end{proposition}

\begin{proof} 1. If $(\alpha, \beta, 0)$ are the VK parameters of a  VK sequence $(\lambda(n))$ with $\lambda(n)_i\geq 0$ for all $i,$  then obviously $\beta \geq 0 $ and $\alpha_1\geq \alpha_2 \geq \ldots \geq 0.$  Moreover, for fixed $N \in \mathbb N$ and $n \geq N$ we have 
$$ \sum_{i=1}^N \alpha_i \,\leq \, \sum_{i=1}^N \Bigl(\alpha_i - \frac{\lambda(n)_i}{n}\Bigr) \, + \, \sum_{i=1}^n \frac{\lambda(n)_i}{n}.$$
As $n \to \infty, $ the first sum tends to $0$ and the second sum tends to $\beta.$ This proves that $\sum_{i=1}^\infty \alpha_i \leq \beta.$

\smallskip

2. Conversely, let $(\alpha, \beta)\in \Omega_+$. In order to construct an associated non-negative VK sequence, we proceed in two steps. 

\smallskip\noindent
(i) Assume that $\alpha$ has at most finitely many non-zero entries. 
If $\alpha \not=0,$ let  $m\in \mathbb N$ be maximal such that $\alpha_i\not= 0$ for $i \leq m.$ If $\alpha = 0,$ let $m:=0.$ Put 
$$ \beta^\prime := \beta - \sum_{i=1}^m \alpha_i \geq 0.$$
For $n >m, $ define $\lambda(n)\in \mathbb R^n$ by
$$ \lambda(n)_i := \begin{cases} n\alpha_i & \text{ if }\, i \leq m  \\
       \frac{n\beta^\prime}{n-m} 	& \text{ if }\, m < i \leq n. 
 \end{cases}$$
Note that the entries of $\lambda(n)$ are non-negative and decreasing for $n$ large enough, say $n \geq n_0$.  It is now straightforward to verify that $(\lambda(n))_{n\geq n_0}$ is a VK sequence with parameters $(\alpha, \beta, 0).$ 

\smallskip\noindent
(ii) Assume that all entries of $\alpha$ are strictly positive. Then a diagonalization argument as in the proof of Proposition \ref{9.12} 
shows that there exists a VK sequence with parameters $(\alpha, \beta, 0).$
\end{proof}

Let us finally turn to consequences in the geometric cases, related to the Cartan motion groups of non-compact Grassmann manifolds. 

For strictly increasing sequences of dimensions $(p_n)_{n\in \mathbb N}, \,(q_n)_{n\in \mathbb N}$ with $ p_n \geq q_n$ consider the sequence of Gelfand  pairs $(G_n, K_n)$ with
\begin{equation}\label{GPB} G_n =(U_{p_n}(\mathbb F)\times  U_{q_n}(\mathbb F))\ltimes M_{p_n,q_n}(\mathbb F), \>\> K_n = U_{p_n}(\mathbb F)\times  U_{q_n}(\mathbb F)\end{equation} 
over $\mathbb F = \mathbb R, \mathbb C, \mathbb H.$ 
It is easily checked that the associated Olshanski spherical pair $(G_\infty, K_\infty)$ is independent of the specific choice of the sequences $(p_n), (q_n),$ and so the same holds for its spherical functions. Indeed, $K= 
U_\infty(\mathbb F) \times U_\infty (\mathbb F)$ and $G_\infty = K \ltimes M_\infty(\mathbb F),$  where $M_\infty(\mathbb F)$ is the space of (in both directions) infinite matrices over $\mathbb F$ with at most finitely many non-zero entries.
The restriction to $q_n = n$ in  Corollary \ref{Cor_B_2} below is therefore not substantial.

	Recall from Section \ref{Bessel} that the positive definite spherical functions of $(G_n, K_n)$, considered as functions on the chamber $\overline C_{q_n} \subset \mathbb R^{q_n} $,  are given by the Bessel functions 
$$  J_{B_{q_n}}(\kappa_n, i\lambda,\,.\,), \> \lambda \in \mathbb R^{q_n} $$
 with the multiplicity 
 $$ \kappa_n = (k^\prime_n, k)= \bigl(\tfrac{d}{2}(p_n-q_n+1) - \tfrac{1}{2}, \tfrac{d}{2}\bigr).$$ So in the geometric cases, the multiplicity parameter $k_n^\prime$ (on the roots $\pm e_i$) naturally varies with $n$. 
 
  We may consider the Olshanski spherical functions of $(G_\infty, K_\infty)$ as $U_\infty(\mathbb F)\times U_\infty(\mathbb F)$-invariant functions on $M_{\infty}(\mathbb F)$ which depend only on the singular values of their argument, or equally as functions on 
 $$\overline C_\infty:= \{ x \in \mathbb R^{(\infty)}: x_1 \geq x_2 \geq \ldots \geq 0\}.$$
 
  The following Corollary goes already back to Pickrell, Theorem 5.14 of \cite{Pi91}, where it was proven by using a slightly different parametrization, and  by different methods. 

\begin{corollary}\label{B_limits}
The set of positive definite spherical functions of the Olshanski spherical pair $(G_\infty, K_\infty),$  considered  as functions on $\overline C_\infty,$ is given by 
$$ \varphi_{(\alpha, \beta)}(x)  = \,
\prod_{j=1}^\infty e^{-\frac{\beta x_j^2}{4}} \prod_{l=1}^\infty \frac{e^{\frac{\alpha_l x_j^2}{4}}}{\bigl(1+ \frac{\alpha_l x_j^2}{2d}\bigr)^{d/2}}, \quad (\alpha, \beta)\in \Omega_+\,.$$
\end{corollary}

\begin{proof} Choose $(G_n, K_n)$ with $q_n = n.$ The proof is then the same as that of Corollary \ref{Cor_A} (1) in the type $A$ case. 
\end{proof}

Let us finally  describe the approximation of the spherical functions of $(G_\infty, K_\infty)$ by spherical functions of an increasing sequence of the Gelfand pairs $(G_n, K_n).$ For simplicity (c.f. the remark above) we now restrict to the case $q_n = n, \, p_n \geq n. $ So in this case, $\,k_n^\prime  = \tfrac{d}{2}(p_n-n+1)-\frac{1}{2}\,$ and $\,\nu_n=\tfrac{d}{2}p_n\,.$

\begin{corollary}\label{Cor_B_2} 
Consider a sequence of Gelfand pairs $(G_n,K_n)$ as in \eqref{GPB} with strictly increasing dimensions $(p_n,q_n),$ where $q_n=n$ and $p_n \geq n.$ 
Assume that  $\varphi_n = J_{B_n}(\kappa_n; i\lambda(n),\,.\,)$ with $\lambda(n) \in \overline C_n$ is a sequence of positive definite spherical functions of $(G_n, K_n)$, considered as
  functions on $\overline C_\infty\supset \overline C_n\,.$  
Then $(\varphi_n)$ converges uniformly on compact subsets of $\,\overline C_\infty$ if and only if
the sequence $\bigl(\frac{\lambda(n)^2}{p_n}\bigr) $ is  Vershik-Kerov. If $(\tfrac{d}{2}\alpha, \tfrac{d}{2}\beta, 0)$ are its VK parameters,  then for $x\in \overline C_\infty,$
$$\lim_{n\to \infty} \varphi_n(x) \,=\, \varphi_{(\alpha, \beta)}(x).$$ 
The convergence is uniform on compact subsets of $\overline C_\infty.$
\end{corollary}

\begin{proof} By use of Theorem \ref{main_B}, the proof is the same as that of Corollary \ref{Cor_A} (2) for the $A$-case.
	\end{proof}

	\begin{remarks} 1. Consider the special case where $q_n = n\,$ and $\,\lim_{n\to \infty} \frac{p_n}{n} =1.$ Then the statement of Corollary \ref{B_limits} remains valid when replacing the   sequence $\bigl(\frac{\lambda(n)^2}{p_n}\bigr)$ by $\bigl(\frac{\lambda(n)^2}{n}\bigr). $
	
	\smallskip
	2. In the group case over $\mathbb F=\mathbb C,$  an approximation
	of the Olshanski spherical functions of $(G_\infty, K_\infty)$ by spherical functions of $(G_n, K_n)$ with $p_n = q_n = n$  was already established in \cite{Ra08}.

		\smallskip
	3. For $\mathbb F = \mathbb C,$ Corollary \ref{B_limits}    is in accordance with  
	results of \cite{Bo19}, where for the semigroup $ \text{Herm}^+  _\infty(\mathbb C)$ of infinite dimensional positive definite matrices over $\mathbb C$, the positive definite Olshanski spherical 
	functions of  $(U_\infty(\mathbb C) \ltimes \text{Herm}^+_\infty(\mathbb C), U_\infty(\mathbb C))$  were determined by semigroup methods and a reduction to the type $A$ case. 
	\end{remarks}

\end{document}